\documentclass[reqno]{amsart}
\usepackage{amssymb,enumerate,mathrsfs,curves}

\numberwithin{equation}{section}

\newtheorem{theorem}[equation]{Theorem}
\newtheorem{proposition}[equation]{Proposition}
\newtheorem{lemma}[equation]{Lemma}
\newtheorem{corollary}[equation]{Corollary}

\theoremstyle{definition}
\newtheorem{definition}[equation]{Definition}
\newtheorem{remark}[equation]{Remark}

\def\C{\mathbb C}
\def\Dom{\mathcal{D}}

\def\K{\mathcal K}
\def\Ka{\tilde{\mathbf K}}
\def\L{\mathscr L}
\def\N{\mathbb N}
\def\R{\mathbb R}
\def\S{\mathscr S}
\def\Sing{\mathcal E}

\def\cev{\hspace*{0.2ex}{}^c\hspace*{-0.2ex} {\mathrm{ev}}}
\def\cg{\hspace*{0.2ex}{}^c\hspace*{-0.2ex} g}
\def\cpi{\hspace*{0.1ex}{}^c\hspace*{-0.1ex} \pi}
\def\csym{\,{}^c\!\sym}
\def\cL{\hspace*{0.2ex}{}^c\hspace*{-0.2ex} L}
\def\cT{\,{}^c T}
\def\dbar{d\hspace*{-0.08em}\bar{\hspace*{0.05em}}}
\def\e{\mathrm e}
\def\embed{\hookrightarrow}
\def\eps{\varepsilon}
\def\m{\mathfrak m}
\def\minus{\backslash}
\def\norm#1{\left\|{#1}\right\|}
\def\open#1{\smash[t]{\overset{{}_{\circ}}{#1}{}}}
\def\set#1{\left\{#1\right\}}
\def\s{\mathrm s}
\def\vp{\varphi}
\def\vt{\vartheta}

\DeclareMathOperator{\bgres}{bg-res}
\DeclareMathOperator{\bgspec}{bg-spec}
\DeclareMathOperator{\Ind}{ind}

\DeclareMathOperator{\Tr}{Tr}
\DeclareMathOperator{\Diff}{Diff}
\DeclareMathOperator{\spec}{spec}
\DeclareMathOperator{\sym}{ \sigma\!\!\!\sigma}

\begin{document}
\title[Trace expansions for elliptic cone operators]{Trace expansions for
elliptic cone operators with stationary domains} 
\author{Juan B. Gil}
\author{Thomas Krainer}
\address{Penn State Altoona\\ 3000 Ivyside Park \\ Altoona, PA 16601-3760}
\author{Gerardo A. Mendoza}
\address{Department of Mathematics\\ Temple University\\ 
Philadelphia, PA 19122}
\begin{abstract}
We analyze the behavior of the trace of the resolvent of an elliptic cone differential operator as the spectral parameter tends to infinity. The resolvent splits into two components, one associated with the minimal extension of the operator, and another, of finite rank, depending on the particular choice of domain. We give a full asymptotic expansion of the first component and expand the component of finite rank in the case where the domain is stationary. The results make use, and develop further, our previous investigations on the analytic and geometric structure of the resolvent. The analysis of nonstationary domains, considerably more intricate, is pursued elsewhere.
\end{abstract}


\maketitle

\section{Introduction}
The present paper is the first of two devoted to the asymptotic expansion of traces of the resolvent as the spectral parameter increases radially, and its consequences for the heat trace and the $\zeta$-functions for general elliptic cone operators acting on sections of a vector bundle $E\to M$ over a compact $n$-dimensional manifold with boundary. 

Let $\Lambda$ be a closed sector in $\C$ and let $A\in x^{-m}\Diff_b^m(M;E)$, $m>0$, see Section~\ref{sec:Preliminaries}. Let $\Dom$ be a domain for $A$ acting in $x^{\gamma} L^2_b(M;E)$, $\gamma\in\R$, and assume that the natural ray conditions on its symbols with respect to $\Lambda$, as discussed in \cite{GKM2}, are satisfied. Then $\Lambda$ is a sector of minimal growth for the extension $A_\Dom$, and for $\ell\in\N$ sufficiently large, $(A_{\Dom}-\lambda)^{-\ell}$ is an analytic family of trace class operators. 

Our aim here and in \cite{GKM5b} is to prove that $\Tr (A_{\Dom}-\lambda)^{-\ell}$ admits a complete asymptotic expansion as $|\lambda|\to\infty$ and to exhibit its generic structure. In this paper we lay the necessary foundations to achieve this goal, and treat in full generality the following special case:

\begin{theorem}\label{ResolventTraceExpansion1}
Suppose $\Dom$ is stationary in the sense of Definition~\ref{def:DwedgeStationary}. Then,
for any $\varphi \in C^{\infty}(M;\textup{End}(E))$ and $\ell\in\N$ with $m\ell>n$, 
\begin{equation*}
\Tr \bigl(\varphi(A_{\Dom}-\lambda)^{-\ell}\bigr) \sim
\sum\limits_{j=0}^{\infty}\sum\limits_{k=0}^{m_j}
\alpha_{jk}\lambda^{\frac{n-j}{m}-\ell}\log^k\lambda
\; \text{ as } |\lambda| \to \infty,
\end{equation*}
with a suitable branch of the logarithm, with constants $\alpha_{jk} \in \C$. 
The numbers $m_j$ vanish for $j < n$, and $m_n \leq 1$. In general, the $\alpha_{jk}$ depend on $\varphi$, $A$, $\Dom$, and $\ell$, but the coefficients $\alpha_{jk}$ for $j < n$ and $\alpha_{n,1}$ do not depend on $\Dom$. If both $A$ and $\varphi$ have coefficients independent of $x$ near $\partial M$, then $m_j = 0$ for all $j>n$.
\end{theorem}
The meaning of being ``independent of $x$'' near the boundary is explained in Definition~\ref{AConstantCoeff} and  Remark~\ref{AandPhiConstantCoeff}.

The special case of the theorem when $\Dom=\Dom_{\min}=x^{m/2}H^m_b(M;E)$ was proved by Loya \cite{LoRes01}. A direct application of our theorem gives the complete short time asymptotics of the heat trace when $A_\Dom$ is sectorial, 
\begin{equation*}
\Tr(\varphi e^{-tA_{\Dom}}) \sim \sum_{j=0}^\infty a_j t^{\frac{j-n}{m}}+ \sum_{j=0}^\infty \sum_{k=0}^{m_j} a_{jk} t^{\frac{j}{m}}\log^{k} t \; \text{ as } t\to 0^+.
\end{equation*}
In the aforementioned special case, this expansion was obtained by Loya \cite{LoRes01} and by Gil \cite{GilThesis,GiHeat01}. Lesch \cite{Le97} obtained a complete expansion of the heat trace in the case of positive selfadjoint extensions with stationary domain assuming that the coefficients are independent of $x$ near the boundary of $M$. For a general positive selfadjoint extension Lesch, op.cit., obtained the partial expansion
\begin{equation*}
\Tr(e^{-tA}) \sim \sum_{j=0}^{n-1} a_j t^{\frac{j-n}{m}}+ O(\log t) \; \text{ as } t\to 0^+.
\end{equation*}

If $\Dom$ is nonstationary, not necessarily selfadjoint, the work presented in this paper gives the partial expansion
\begin{equation*}
\Tr \bigl(\varphi(A_{\Dom}-\lambda)^{-\ell}\bigr) \sim
\sum\limits_{j=0}^{n-1}\alpha_{j,0}\lambda^{\frac{n-j}{m}-\ell}
+\alpha_{n,1}\lambda^{-\ell}\log\lambda + O(|\lambda|^{-\ell})
\end{equation*}
as $|\lambda| \to \infty$. This implies
\begin{equation*}
\Tr(\varphi e^{-tA_{\Dom}}) \sim \sum_{j=0}^{n-1} a_j t^{\frac{j-n}{m}}+ 
a_{n,1}\log t + O(1)\; \text{ as } t\to 0^+
\end{equation*}
if $A_\Dom$ is sectorial. As mentioned before, a complete asymptotic expansion of the resolvent and the heat trace for nonstationary domains will be discussed in \cite{GKM5b}.

\medskip
Another consequence of Theorem~\ref{ResolventTraceExpansion1} is that, if $A$ is bounded from below on the minimal domain, then the $\zeta$-function of any selfadjoint extension with stationary domain (e.g. the Friedrichs extension) has a meromorphic extension to all of $\C$ with poles contained in the set
\begin{equation}\label{PolesLocation}
\set{\frac{n-k}{m}:k\in \N_0}.
\end{equation}

An example worked out by Falomir, Pisani, and Wipf \cite{FPW} (an ODE on a half-line) yielded a $\zeta$-function extending meromorphically to all of $\C$ but exhibiting additional poles at points not in the set \eqref{PolesLocation}. The problem was taken up again by Falomir, Muschietti, Pisani, and Seeley \cite{FMPSeeley}, this time analyzing a first order system of ODEs on an interval, then for partial differential operators of Laplace type (with coefficients independent of $x$) by Kirsten, Loya, and Park \cite{KLP06,KLP08a,KLP08b}; their work shows that there may be logarithmic singularities of the $\zeta$-function that impede its meromorphic extension to all of $\C$. Our results show that these pathologies arise only when the domain is nonstationary. 

\medskip
Conic metrics are typical examples of incomplete Riemannian metrics, so the corresponding Laplacians typically admit many selfadjoint extensions in $x^{\gamma} L^2_b$. Ever since Cheeger's seminal paper \cite{Ch79}, the predominant method used in the spectral theory of these extensions has been a separation of variables ansatz, together with a pasting argument away from the singularities. While Cheeger's results have been extended and generalized to various classes of operators, see e.g. \cite{BrSe85,BrSe87,BrSe91,Ca83, FMP04, FMPSeeley, Le97, Mooers}, the method of separation of variables (combined with sophisticated ODE techniques and special functions) has remained at the core of and has limited almost all subsequent investigations in the field. For instance, the heat kernel and $\zeta$-function of an arbitrary selfadjoint extension of $\Delta_{\textup{warped}}$, cf. Section~\ref{sec:Preliminaries}, are to this day not understood, except for the Friedrichs extension, see \cite{BrSe87, BrSe91, GiHeat01, Le97, LoRes01}.

In this paper we make use of, and extend, the analytic and geometric methods developed in \cite{GKM1,GKM2,GKM3}. Our results can be applied to selfadjoint and nonselfadjoint extensions of $\Delta_{\textup{warped}}$, and more generally, to extensions of $\Delta_{\cg}$ on functions as well as on \mbox{($c$-)forms}, see Section~\ref{sec:Preliminaries}. In fact, the statements of Theorem~\ref{ResolventTraceExpansion1} are valid particularly for the Friedrichs extension of any semibounded cone operator. 

\medskip
The structure of the paper is as follows. In the next section we will set up the terminology and will review some facts about cone operators, their symbols, and their closed extensions in a reference weighted $L^2$ space modeled by the underlying cone geometry. We will discuss specific examples of Laplacians associated with a cone metric and will use them to illustrate the notion of coefficients independent of $x$ near the boundary. We will also review the model operator, its scaling properties, the associated domain, and the concept of being stationary. 

In Section~\ref{sec:ResolventStructure} we review the structure of the resolvent and describe its various components according to the decomposition
\begin{equation*}
 (A_\Dom-\lambda)^{-1} = B(\lambda) + G_{\Dom}(\lambda)
\end{equation*}
given in Theorem~\ref{ResolventStructure}. In that section we also summarize the proof of our main result which is then carried out in the subsequent sections. In Section~\ref{sec:DminAsymptotics} we obtain the trace asymptotics associated with $B(\lambda)$, and in Section~\ref{sec:FiniteContribution} we discuss the asymptotic properties of $G_{\Dom}(\lambda)$. 

\section{Preliminaries}
\label{sec:Preliminaries}

\def\Gr{\mathrm{Gr}}

Cone operators arise when studying partial differential equations on manifolds with conical singularities. By introducing spherical coordinates centered at the singularities, the equations exhibit a typical degeneracy in the radial variable. Topologically, a manifold with conical singularities can be realized as a quotient $M/{\sim}$ of a smooth compact $n$-manifold $M$ with boundary $Y = \partial M$ with respect to an equivalence relation that collapses boundary components to points. Both geometric and analytic investigations associated with such a configuration actually take place on the manifold $M$.

\subsection*{Cone geometry}
The natural framework for cone geometry is the $c$-cotangent bundle
\begin{equation}\label{cCotangent}
\cpi:\cT^*M\to M,
\end{equation}
see~\cite{GKM1}, a vector bundle whose space of smooth sections is in one-to-one correspondence with that consisting of all smooth $1$-forms on $M$ that are conormal to $Y$, i.e., all $\omega \in C^{\infty}(M,T^*M)$ whose pullback to $Y$ vanishes. The isomorphism is given by a bundle homomorphism
\begin{equation}\label{TheHomomorphism}
\cev: \cT^*M \to T^*M
\end{equation}
which is an isomorphism over $\open M$. In coordinates $(x,y_1,\dotsc,y_{n-1})$ near the boundary, where $x$ is a defining function for $Y$, a local frame for $\cT^*M$ is given by the sections mapped by $\cev$ to the forms $dx$, $xdy_1, \dotsc, xdy_{n-1}$. As is customary in the context of analysis on manifolds with singularities, the boundary defining function $x$ is always assumed to be positive in $\open M$. 

By a \emph{$c$-metric} we mean any metric on the dual of $\cT^*M$. Such a metric induces (via the homomorphism \eqref{TheHomomorphism}) a Riemannian metric $\cg$ on $\open M$. In coordinates near the boundary as in the previous paragraph, $\cg$  is represented as a smooth symmetric $2$-cotensor
\begin{equation*}
\cg = g_{00}\,dx\otimes dx + \sum\limits_{j=1}^{n-1}g_{0j}\,dx\otimes xdy_j +
\sum\limits_{i=1}^{n-1}g_{i0}\,xdy_i\otimes dx + \sum\limits_{i,j=1}^{n-1}g_{ij}\,xdy_i\otimes xdy_j;
\end{equation*}
the matrix $(g_{ij})$ depends smoothly on $(x,y)$ and is positive definite up to $x=0$.

Special cases of $c$-metrics are warped and straight cone metrics. A warped cone metric is a Riemannian metric on $M$ such that there is a diffeomorphism of a neighborhood $U$ of $Y$ in $M$ to $[0,\eps) \times Y$ under which the metric takes on the form $dx^2 + x^2g_Y(x)$ for a family of metrics $g_Y(x)$ on $Y$ which is smooth up to $x=0$; here $x$ is of course the variable in $[0,\eps)$. If the diffeomorphism is such that $g_Y(x)$ is in fact independent of $x$ for small $\eps$, then $\cg$ is a straight cone metric. The study of spectral theory on manifolds equipped with straight cone metrics was initiated in Cheeger's paper \cite{Ch79}.

In coordinates near the boundary, the positive Laplacian on functions corresponding to a general $c$-metric $\cg$ takes the form
\begin{equation*}
\Delta_{\cg} = x^{-2}\sum\limits_{k+|\alpha| \leq 2}a_{k\alpha}(x,y)(xD_x)^kD_y^{\alpha}
\end{equation*}
with coefficients smooth up to $x=0$. In particular, for a straight cone metric, 
\begin{align*}
\Delta_{\textup{straight}} &= x^{-2}\Bigl((xD_x)^2 - i(n-2)(xD_x) + \Delta_{g_Y}\Bigr), 
\intertext{and for a warped cone metric,}
\Delta_{\textup{warped}} &= x^{-2}\Bigl((xD_x)^2 - a(x,y)(xD_x) + \Delta_{g_Y(x)}\Bigr),
\end{align*}
where $a(x,y) = i\bigl((n-2)+|g_Y(x)|^{-1/2}  x\partial_x |g_Y(x)|^{1/2}\bigr)$. Notice that, unlike these two special cases, the general Laplacian $\Delta_{\cg}$ may contain mixed derivatives with respect to $x$ and the variables $y_1,\dots,y_{n-1}$. 

\subsection*{Cone operators}
Let $E\to M$ be a smooth vector bundle. A cone differential operator of order $m$ on sections of $E$ is an element $A$ of $x^{-m}\Diff_b^m(M;E)$, where $\Diff_b^m(M;E)$ is the space of totally characteristic differential operators of order $m$, see \cite{RBM2}. Thus $A$ is a linear differential operator on $C^\infty(\open M;E)$, of order $m$, which near any point in $Y$, in coordinates as above, is of the form 
\begin{equation}\label{cone-operator}
  A=x^{-m}\sum_{k+|\alpha|\le m} a_{k\alpha}(x,y)(xD_x)^k D_y^\alpha 
\end{equation} 
with coefficients $a_{k\alpha}$ smooth up to $x=0$.

\begin{definition}\label{AConstantCoeff}
We say that $A$ has \emph{coefficients independent of $x$ near $Y$} if there is a diffeomorphism $\Phi$ of a neighborhood $U$ of $Y$ in $M$ to $Y\times [0,\eps)$ such that with the canonical projection $\pi: Y\times [0,\eps)\to Y$, with $x$ the variable on $[0,\eps)$, and with some isomorphism $E|_U\to \pi^*(E|_Y)$ covering $\Phi$ we have $A=x^{-m} \sum A_\ell (xD_x)^\ell$ where $A_\ell\in \Diff^{m-\ell}(Y;E|_Y)$.
\end{definition}

Having coefficients independent of $x$ means in effect that $A$ can be analyzed near the boundary using separation of variables. According to this definition, $\Delta_{\textup{straight}}$ has coefficients independent of $x$ near $Y$. In general, this is not the case for $\Delta_{\textup{warped}}$.

\begin{remark}\label{AandPhiConstantCoeff}
An element $\varphi\in C^\infty(M;\textup{End}(E))$ can be viewed as a differential operator of order $0$, so the definition above may be applied to such $\varphi$. In the last assertion of Theorem \ref{ResolventTraceExpansion1} we mean that $A$ and $\varphi$ have coefficients independent of $x$ near $Y$ with respect to the same trivialization. 
\end{remark}

The standard principal symbol of $A$ over the interior determines, with the aid of the map $\cev$ in \eqref{TheHomomorphism}, a smooth homomorphism $\cpi^*E\to\cpi^*E$. This is the $c$-principal symbol $\csym(A)$ of $A$. In local coordinates near $Y$, 
\begin{equation*}
  \csym(A)=\sum_{k+|\alpha|= m} a_{k\alpha}(x,y)\xi^k \eta^\alpha. 
\end{equation*} 
The operator $A$ is said to be $c$-elliptic if $\csym(A)$ is invertible on $\cT^*M\minus 0$.

Associated with $A=x^{-m}P$ there is an operator-valued family
\begin{equation*}
 \C\ni \sigma\mapsto \hat P(\sigma)\in \Diff^m(Y;E|_Y),
\end{equation*}
which in accordance to \eqref{cone-operator} can be represented as
\begin{equation*}
 \hat P(\sigma) = \sum_{k+|\alpha|\le m} a_{k\alpha}(0,y)\sigma^k D_y^\alpha
\end{equation*}
and is called the conormal symbol of $A$ (or of $P$). If $A$ is $c$-elliptic, then $\hat P(\sigma)$ is invertible for all $\sigma\in\C$ except a discrete set 
\begin{equation*}
 \spec_b(A)=\set{\sigma\in\C: \hat P(\sigma) \text{ is not invertible}}
\end{equation*}
called the boundary spectrum of $A$, see \cite{RBM2}.

\subsection*{Closed extensions and Fredholmness}

Let $E$ be a vector bundle over $M$. Then $\cL^2(M;E)$ is the space of $L^2$ sections of $E$ with respect to some Hermitian metric on $E$ and the Riemannian density determined by some $c$-metric on $M$. The space and its (locally convex) topology are independent of the choice of Hermitian metric and Riemannian density, but of course if selfadjointess is of interest then these metrics are fixed. 

The density induced by a $c$-metric is of the form $x^{n-1}\m$ for a smooth positive density $\m$ on $M$, where $n=\dim M$. In other words,
\begin{equation*}
 \cL^2(M;E) = x^{-n/2} L^2_b(M;E),
\end{equation*}
where $L^2_b(M;E)$ is the $L^2$ space with respect to the $b$-density $\frac{1}{x}\m$, see \cite{RBM2}.

The starting point of the analysis of $A$ is the densely defined unbounded operator
\begin{equation}\label{StartDomain}
A:C_c^\infty(\open M;E)\subset x^{\mu} L^2_b(M;E)\to x^{\mu} L^2_b(M;E)
\end{equation}
for some $\mu\in \R$. One may assume (as we will here) that $\mu=-m/2$, since this situation can be attained by replacing $A$ with $x^{-\mu-m/2} A\,x^{\mu+m/2}\in x^{-m}\Diff_b^m(M;E)$. This particular choice of weight has technical advantages when dealing with adjoints, see \cite{GiMe01}.

There are two canonical closed extensions of $A$: 
\begin{align*}
\Dom_{\min} &= \text{domain of the closure of \eqref{StartDomain}},\\
\Dom_{\max} &=\{u\in x^{-m/2}L^2_b(M;E): Au\in x^{-m/2}L^2_b(M;E)\}.
\end{align*}
These are complete with respect to the graph norm, $\|u\|_A=\|u\|+\|Au\|$, and the former is a subspace of the latter. Both domains are dense in $x^{-m/2}L^2_b(M;E)$. The following theorem implies that if $A$ is $c$-elliptic, then any intermediate space $\Dom$
\begin{equation*}
\Dom_{\min}\subset \Dom\subset \Dom_{\max},
\end{equation*}
gives rise to a closed extension 
\begin{equation*}
A_{\Dom}:\Dom\subset x^{-m/2}L^2_b(M;E)\to x^{-m/2}L^2_b(M;E).
\end{equation*}

\begin{theorem}[Lesch \cite{Le97}] \label{LeschTheorem}
If $A\in x^{-m}\Diff_b^m(M;E)$ is $c$-elliptic, then
\begin{equation*}
\dim \Dom_{\max}/\Dom_{\min}<\infty
\end{equation*}
and all closed extensions of $A$ are Fredholm.
Moreover, 
\begin{equation}\label{RelIndexA}
\Ind A_{\Dom}= \Ind A_{\Dom_{\min}}+ \dim\Dom/\Dom_{\min}. 
\end{equation}
\end{theorem}

More details about the structure of domains will be given below.

\subsection*{The model operator}
The model operator of an element $A\in x^{-m}\Diff_b^m(M;E)$ is an invariantly defined operator $A_\wedge$ on the half line bundle $\pi:N_+Y\to Y$, the closed inward normal bundle of $Y$, that in local coordinates takes the form  
\begin{equation*}
 A_\wedge=x^{-m}\sum_{k+|\alpha|\le m} a_{k\alpha}(0,y)(xD_x)^k D_y^\alpha,
\end{equation*}
if $A$ is given by \eqref{cone-operator}. The operator $A_\wedge$ is $c$-elliptic if $A$ is.

We trivialize $N_+Y$ as $Y^\wedge=[0,\infty)\times Y$ using the defining function $x$ and denote the variable in $[0,\infty)$ also by $x$. For simplicity we often write $E$ instead of $\pi^*(E|_Y)$. Let $L^2_b(Y^\wedge;E)$ be the $L^2$ space with respect to a density of the form $\frac{dx}{x}\otimes \pi^*\mathfrak{m}_Y$ and the canonically induced Hermitian form on $\pi^*(E|_Y)$; $\m_Y$ is any smooth positive density on $Y$.

The operator $A_\wedge\in x^{-m}\Diff_b^m(Y^\wedge;E)$ acts on $C^\infty_c(\open Y^\wedge;E)$. Just as with $A$ there are two canonical extensions to closed densely defined operators on the space $x^{-m/2}L^2_b(Y^\wedge;E)$; naturally, their domains are denoted $\Dom_{\wedge,\min}$ and $\Dom_{\wedge,\max}$. Like for $A$, the space $\Dom_{\wedge,\max}/ \Dom_{\wedge,\min}$ is finite dimensional, in fact there is an important natural isomorphism 
\begin{equation*}
\theta:\Dom_{\max}/ \Dom_{\min}\to \Dom_{\wedge,\max}/ \Dom_{\wedge,\min} 
\end{equation*}
discussed below.

The model operator exhibits a fundamental invariance property with respect to the natural $\R_+$-action on $Y^\wedge$. Let
\begin{equation*}
\R_+\ni \varrho \mapsto
\kappa_\varrho:x^{-m/2}L^2_b(Y^\wedge;E)\to x^{-m/2}L^2_b(Y^\wedge;E)
\end{equation*}
be the one-parameter group of isometries defined by 
\begin{equation}\label{kapparho}
  (\kappa_\varrho f)(x,y)= \varrho^{m/2} f(\varrho x,y).
\end{equation}
This definition, which appears to treat $f$ as a function, has an obvious interpretation for sections of $\pi^*(E|_Y)\to Y^\wedge$. Since $A_\wedge$ satisfies
\begin{equation}\label{kappaHomogeneous0}
  \kappa_\varrho A_\wedge =\varrho^{-m} A_\wedge\kappa_{\varrho},
\end{equation}
the canonical domains $\Dom_{\wedge,\min}$ and $\Dom_{\wedge,\max}$ are both $\kappa$-invariant. In particular, $\kappa$ induces an action on $\Dom_{\wedge,\max}/\Dom_{\wedge,\min}$, therefore an action on each of the Grassmannian manifolds associated with this quotient:
\begin{equation}\label{ActionOnGrassmannian}
\kappa:\R_+\times \Gr_{d''}(\Dom_{\wedge,\max}/\Dom_{\wedge,\min})\to \Gr_{d''}(\Dom_{\wedge,\max}/\Dom_{\wedge,\min}).
\end{equation}
Also from \eqref{kappaHomogeneous0} we get
\begin{equation}\label{kappaHomogeneous}
  A_\wedge-\varrho^m\lambda = \varrho^{m} \kappa_\varrho (A_\wedge-\lambda) 
  \kappa_{\varrho}^{-1} 
\end{equation}
for every $\varrho>0$ and $\lambda\in\C$. This property is called $\kappa$-homogeneity,  see e.g. \cite{SchulzeNH}.

As with $A$, any intermediate space $\Dom_\wedge$ with $\Dom_{\wedge,\min}\subset \Dom_\wedge\subset \Dom_{\wedge,\max}$ gives rise to a closed extension
\begin{equation*}
A_{\wedge,\Dom_\wedge}:\Dom_\wedge\subset x^{-m/2}L^2_b(Y^\wedge;E)
\to x^{-m/2}L^2_b(Y^\wedge;E).
\end{equation*}
We define the background spectrum and the background resolvent set of $A_\wedge$ as
\begin{gather*}
 \bgspec A_\wedge = \!\! \bigcap_{\Dom_{\wedge,\min}\subset \Dom_\wedge\subset \Dom_{\wedge,\max}} \!\!\spec A_{\wedge,\Dom_\wedge}, \\
 \bgres A_\wedge = \C\minus \bgspec A_\wedge.
\end{gather*}
Clearly $\bgspec A_\wedge$ is closed, and using the $\kappa$-homogeneity one obtains easily that $\bgres A_\wedge$ is a union of open sectors. Furthermore, if $\lambda\in \bgres A_\wedge$, then $A_{\wedge,\Dom_\wedge} -\lambda$ is Fredholm and
\begin{equation*}
 \Ind(A_{\wedge,\Dom_\wedge}-\lambda)=
 \Ind(A_{\wedge,\min}-\lambda)+\dim\Dom_\wedge/\Dom_{\wedge,\min},
\end{equation*}
see \cite[Section~7]{GKM1}. The index is constant on connected components of $\bgres A_\wedge$.

\subsection*{Stationary domains and the isomorphism $\boldsymbol{\theta}$}
Let
\begin{equation*}
\Sigma=\spec_b(A)\cap \{\sigma\in\C: -m/2<\Im\sigma<m/2\}.
\end{equation*}
For every $\sigma_0\in \Sigma$ we let $\Sing_{\wedge,\sigma_0}$ be the space of singular functions of the form 
\begin{equation*}
 \psi=x^{i\sigma_0}\sum_{k=0}^{\mu_{\sigma_0}} c_{\sigma_0,k}(y)\log^k x 
 \;\text{ with } c_{\sigma_0,k}\in C^\infty(Y;E)
\end{equation*}
such that $A_\wedge \psi=0$. Since $A$ is assumed to be $c$-elliptic, $\Sing_{\wedge,\sigma_0}$ is finite dimensional. Also,
\begin{equation*}
 \Dom_{\wedge,\max}/\Dom_{\wedge,\min} \cong \bigoplus_{\sigma_0\in\Sigma} \Sing_{\wedge,\sigma_0}\subset C^\infty(\open Y^\wedge ;E)
\end{equation*}
with the isomorphism given by the map
\begin{equation}\label{wedgequotientIso}
 \psi\mapsto (\omega \psi+\Dom_{\wedge,\min}): \bigoplus_{\sigma_0\in\Sigma} \Sing_{\wedge,\sigma_0} \to \Dom_{\wedge,\max}/\Dom_{\wedge,\min}
\end{equation}
for an arbitrary cut-off function $\omega\in C_c^\infty([0,1))$ (a function which equals $1$ in a neighborhood of the origin).

There are analogous spaces $\Sing_{\sigma_0}\subset C^\infty(\open Y^\wedge;E)$ and isomorphism
\begin{equation}\label{quotientIso}
\Dom_{\max}/\Dom_{\min} \to \bigoplus_{\sigma_0\in\Sigma} \Sing_{\sigma_0} \subset C^\infty(\open Y^\wedge ;E)
\end{equation}
associated with $A$. The space $\Sing_{\sigma_0}$ is canonically isomorphic to $\Sing_{\wedge,\sigma_0}$ and consist of the functions of the form
\begin{equation*}
\sum_{\vartheta=0}^{m} \Bigl(\sum_{k=0}^{K_{\sigma_0-i\vartheta}} c_{\sigma_0-i\vartheta,k}(y) \log^kx\Bigr)x^{i(\sigma_0-i\vartheta)}\;\text{ with } c_{\sigma_0-i\vartheta,k}\in C^\infty(Y;E).
\end{equation*}
The space and the isomorphism are defined as follows. First, the operator $A$ has a Taylor expansion near $Y$,
\begin{equation*}
 A \sim x^{-m} \sum_{\nu=0}^\infty P_\nu x^\nu,
\end{equation*}
where each $P_k\in \Diff^m_b(Y^\wedge;E)$ has coefficients independent of $x$. Implicit in this expansion is the fact that one made a choice of trivialization $\pi:[0,\eps)\times Y\to Y$ of a collar neighborhood of $Y$ and an isomorphism $E\cong\pi^*(E|_Y)$ near $Y$. The following makes use of these trivializations. Let $\hat P_k(\sigma)$ be the conormal symbol (indicial family) associated with $P_k$.   
Define linear operators
\begin{equation}\label{eSigmak}
 \e_{\sigma_0,\vartheta}:\Sing_{\wedge,\sigma_0}\to C^\infty(\open Y^\wedge;E),\quad \vartheta=0,1,2,\dotsc
\end{equation}
inductively as follows:
\begin{enumerate}[($i$)]
\item $\e_{\sigma_0,0}$ is the identity map. 
\item Given $\e_{\sigma_0,0},\dotsc,\e_{\sigma_0,\vartheta-1}$ for some $\vartheta \in \N$, define $\e_{\sigma_0,\vartheta}(\psi)$ for $\psi \in \Sing_{\wedge,\sigma_0}$ to be the unique function of the form
\begin{equation*}
\Bigl(\sum_{k=0}^{m_{\sigma_0-i\vartheta}}c_{\sigma_0-i\vartheta,k}(y)\log^kx\Bigr) x^{i(\sigma_0-i\vartheta)}
\end{equation*}
such that
\begin{equation*}
\qquad
(\omega \e_{\sigma_0,\vartheta}(\psi))^{\wedge}(\sigma) + 
\hat{P}_0(\sigma)^{-1}\Bigl(\sum_{k=1}^{\vartheta}\hat{P}_k(\sigma)
\s_{\sigma_0-i\vartheta}(\omega \e_{\sigma_0,\vartheta-k}(\psi))^{\wedge}
(\sigma+ik)\Bigr)
\end{equation*}
is holomorphic at $\sigma = \sigma_0 - i\vartheta$, where 
\begin{equation*}
(\omega \e_{\sigma_0,\vartheta-k}(\psi))^{\wedge}(\sigma)
= \int_0^\infty x^{-i\sigma} \omega \e_{\sigma_0,\vartheta-k}(\psi) \frac{dx}{x}
\end{equation*}
is the Mellin transform of $\omega \e_{\sigma_0,\vartheta-k}(\psi)$, and $\s_{\sigma_0-i\vartheta}(\omega \e_{\sigma_0,\vartheta-k}(\psi))^{\wedge} (\sigma+ik)$ is the singular part of its Laurent expansion at $\sigma_0 - 
i\vartheta$. Here, $\omega \in C_c^{\infty}(\overline{\R}_+)$ is an arbitrary 
cut-off function near zero.
Observe that the Mellin transform of $\omega \e_{\sigma_0,\vartheta-k}(\psi)$ 
is meromorphic in $\C$ with pole only at $\sigma_0 - i(\vartheta-k)$.
\end{enumerate}
The map
\begin{equation*}
\sum_{\vartheta=0}^m \e_{\sigma_0,\vartheta}:\Sing_{\wedge,\sigma_0}\to C^\infty(\open Y^\wedge;E)
\end{equation*}
is injective; letting $\Sing_{\sigma_0}$ be its image we get an isomorphism 
\begin{equation*}
\theta^{-1}_{\sigma_0}:\Sing_{\wedge,\sigma_0}\to \Sing_{\sigma_0}
\end{equation*}
The maps $\theta_{\sigma_0}^{-1}$ together with the isomorphisms \eqref{wedgequotientIso} and \eqref{quotientIso} determine an isomorphism
\begin{equation*}
 \theta: \Dom_{\max}/\Dom_{\min}\to \Dom_{\wedge,\max}/\Dom_{\wedge,\min}
\end{equation*}
which in turn establishes a natural correspondence between subspaces of the domain and the range.

\begin{definition}\label{def:DwedgeStationary}
Let $\Dom$ be a domain for $A$. 
\begin{enumerate}[$(i)$]
\item The domain $\Dom_{\wedge}$ for $A_\wedge$ defined via
\begin{equation*}
\Dom_{\wedge}/\Dom_{\wedge,\min}= \theta\big(\Dom/\Dom_{\min}\big)
\end{equation*}
is called the \emph{associated domain} of $\Dom$, see \cite{GKM1,GKM2}.
\item The domain $\Dom$ is said to be \emph{stationary} if its associated domain is $\kappa$-invariant.
\end{enumerate}
\end{definition}

It is a quite common assumption in the existing literature that the domain $\Dom$ for $A$ is invariant under the apparently natural action of $\kappa$ on the elements of $\Sing = \bigoplus_{\sigma_0\in \Sigma} \Sing_{\sigma_0}$ (via the isomorphism \eqref{quotientIso}). However, this is an assumption that makes sense only when $A$ has coefficients independent of $x$ near $Y$. In \cite[Section~5]{GKM3} we gave a very simple example for which $\Dom_{\max}$ (or $\Sing$) is not $\kappa$-invariant, which means that allowing $\kappa$ to act directly on the domain of $A$ on the manifold $M$ is conceptually incorrect (this is so even when discussing general closed extensions of $\Delta_{\textup{warped}}$). The map $\theta$ resolves the issue by expressing the condition in terms of domains for $A_\wedge$, for which scaling is indeed natural.

Stationary domains other than $\Dom_{\min}$ and $\Dom_{\max}$ (if these spaces are different) do exist in general. For instance, by the results in \cite{GiMe01}, it is easy to see that the Friedrichs extension of any semibounded cone operator is stationary. This applies in particular to arbitrary $c$-Laplacians $\Delta_{\cg}$.

More generally, let $d=\dim(\Dom_{\wedge,\max}/\Dom_{\wedge,\min})$ and assume $d>1$. Let $0<d''<d$, and let $\mathcal G=\Gr_{d''}(\Dom_{\wedge,\max}/\Dom_{\wedge,\min})$. The infinitesimal generator of the action \eqref{ActionOnGrassmannian} is a real (and real-analytic) vector field $\mathcal T$ (see \cite{GKM1}). Since $\dim \mathcal G \ne 0$, the Euler characteristic of $\mathcal G$ is not zero, so $\mathcal T$ must vanish somewhere. Those points in $\mathcal G$ where $\mathcal T$ vanishes correspond in an obvious manner to $\kappa$-invariant domains $\Dom_{\wedge,\min}\subset \Dom_\wedge\subset \Dom_{\wedge,\max}$.

\section{Structure and expansion of the resolvent}
\label{sec:ResolventStructure}

A cut-off function $\omega \in C_c^{\infty}([0,1))$ is a function which equals $1$ in a neighborhood of the origin. We will consider such a function as a function on both $M$ and $Y^\wedge$, supported in a collar neighborhood of the boundary $Y=\partial M$. We will use the notation $\phi \prec \psi$ to indicate that the function $\psi$ equals $1$ in a neighborhood of the support of the function $\phi$; in particular, $\phi\psi =\vp$.

Let $\Lambda$ be a closed sector in $\C$ of the form
\begin{equation*}
\Lambda=\set{\lambda\in\C: \lambda=re^{i\theta}\text{ for } r\ge 0, \ |\theta-\theta_0|\le a}
\end{equation*}
for some real $\theta_0$ and $a>0$. For $R>0$ we denote
\begin{equation*}
 \Lambda_R = \set{\lambda\in\Lambda: |\lambda|\ge R}
\end{equation*}

In the following theorem, proved in \cite[Section 6]{GKM2}, and in all subsequent sections here, we assume that $A\in x^{-m}\Diff_b^m(M;E)$, $m>0$, is $c$-elliptic with parameter in $\Lambda$, that is,
\begin{equation*}
\csym(A)-\lambda \;\text{ is invertible on } \big(\cT^*M\times\Lambda\big)\minus 0. 
\end{equation*}
We consider closed extensions $A_\Dom$ of $A$ in $x^{-m/2}L^2_b(M;E)$ such that $\Lambda$ is a sector of minimal growth for $A_{\wedge,\Dom_\wedge}$, where $\Dom_\wedge$ is the associated domain of $\Dom$. 

Note that, as shown in \cite{GKM1}, if $\Dom$ is stationary, then $\Lambda$ is a sector of minimal growth for $A_{\wedge,\Dom_\wedge}$ if and only if
\begin{equation*}
\Lambda\minus\{0\}\subset \bgres(A_\wedge) \text{ and } A_{\wedge,\Dom_{\wedge}}-\lambda_0 \text{ is invertible for some } \lambda_0\in \Lambda\minus\{0\}.
\end{equation*}

\begin{theorem}\label{ResolventStructure}
Under the previous assumptions on the symbols of $A$, we have that $\Lambda$ is a sector of minimal growth for the extension $A_{\Dom}$, and there is $R>0$ such that 
\begin{equation*}
 (A_\Dom-\lambda)^{-1} = B(\lambda) + G_{\Dom}(\lambda)
 \;\text{ for every } \lambda\in\Lambda_R.
\end{equation*}
Here $B(\lambda)$ is a certain parametrix of $A-\lambda$ with $B(\lambda)(A-\lambda)|_{\Dom_{\min}}=1$ for $\lambda\in\Lambda_R$, and $G_{\Dom}(\lambda)$ is a smoothing pseudodifferential operator of finite rank. 
\end{theorem}

The parametrix $B(\lambda)$ is of the form (cf. \cite[Section~5]{GKM2})
\begin{equation}\label{DminParametrix}
 B(\lambda)= \tilde\omega Q(\lambda)\tilde\omega_1 + (1-\tilde\omega)Q_{\rm int}(\lambda) (1-\tilde\omega_0)+ G(\lambda)
\end{equation}
for some cut-functions $\tilde\omega$, $\tilde\omega_0$, $\tilde\omega_1\in C^\infty_0([0,1))$ with $\tilde\omega_0\prec \tilde\omega\prec \tilde\omega_1$, where: 
\begin{enumerate}[$(i)$]
 \item $G(\lambda)$ is a smooth family of smoothing operators that, together with its derivatives, admits an asymptotic expansion (as $|\lambda|\to\infty$) similar to the expansion in Lemma~\ref{BSymbolExpansion}; 
 \item $Q_{\rm int}(\lambda)$ is a standard parameter-dependent parametrix of $A-\lambda$ over $\open M$;
 \item $Q(\lambda)$ is a Mellin operator defined by
\begin{equation}\label{Bminlambda}
 Q(\lambda)u(x)=\frac{1}{2{\pi}}\int\limits_{\R}
\int\limits_{(0,1)}\Bigl(\frac{x}{x'}\Bigr)^{i\sigma-\frac{m}{2}} x^m h(x,\sigma,x^m\lambda) u(x')\,\frac{dx'}{x'}\,d\sigma,
\end{equation}
for $u\in C_c^\infty((0,1),C^\infty(Y;E))$, where $h$ is an operator-valued parameter-dependent Mellin symbol of order $-m$. What we need to know here is that in a local patch $\Omega\subset Y$, the family $h(x,\sigma,\lambda)$ can be expressed by a symbol $p(x,y,\eta,\lambda)$ with $(x,y)\in \overline{\R}_+\times \Omega$, $\eta=(\sigma,\xi)\in \R\times\R^{n-1}$, and $\lambda\in\Lambda$, that admits an asymptotic expansion 
\[ p\sim \sum_{k=0}^\infty p_k \]
such that for $|\eta|+|\lambda|^{1/m}\ge 1$,
\begin{equation*}
 p_k(x,y,t\eta,t^m\lambda)=t^{-m-k} p_k(x,y,\eta,\lambda)
 \text{ for every } t\ge 1.
\end{equation*}
\end{enumerate}

The construction of $G_\Dom(\lambda)$ follows a ``reduction to the boundary'' approach that we proceed to describe briefly. Under our general assumptions, there is an operator family $K(\lambda):\C^d\to x^{-m/2}L^2_b$, with $d=-\Ind A_{\Dom_{\min}}$, such that
\begin{equation*}
\begin{pmatrix} (A-\lambda)|_{\Dom_{\min}} & K(\lambda) \end{pmatrix}:
\begin{array}{c} \Dom_{\min}\\ \oplus\\ \C^d \end{array} \to x^{-m/2}L^2_b
\end{equation*}
is invertible for $\lambda\in\Lambda_R$ for some $R>0$. The inverse can be written as
\begin{equation*}
 \begin{pmatrix} (A-\lambda)|_{\Dom_{\min}} & K(\lambda) \end{pmatrix}^{-1}= \binom{B(\lambda)}{T(\lambda)},
\end{equation*}
where $B(\lambda)$ is the parametrix of $A-\lambda$ on $\Dom_{\min}$, and $T(\lambda):x^{-m/2}L^2_b\to \C^d$ has the properties listed in Proposition~\ref{Tlambda}.
If we split $\Dom=\Dom_{\min}\oplus \tilde\Sing$ and write 
\begin{equation*}
 A_\Dom-\lambda =
 \begin{pmatrix} (A-\lambda)|_{\Dom_{\min}} & (A-\lambda)|_{\tilde\Sing} \end{pmatrix},
\end{equation*}
then
\begin{equation*}
\binom{B(\lambda)}{T(\lambda)}
\begin{pmatrix} (A-\lambda)|_{\Dom_{\min}} & (A-\lambda)|_{\tilde\Sing} \end{pmatrix}=
\begin{pmatrix}
 B(\lambda)(A-\lambda)|_{\Dom_{\min}} & B(\lambda)(A-\lambda)|_{\tilde\Sing} \\
 0 & T(\lambda)(A-\lambda)|_{\tilde\Sing}
\end{pmatrix},
\end{equation*}
so $A_\Dom-\lambda$ is invertible if and only if $T(\lambda)(A-\lambda): \tilde\Sing\to \C^d$ is invertible. 

By construction, $T(\lambda)(A-\lambda):\Dom_{\max}(A)\to \C^d$ vanishes on $\Dom_{\min}(A)$, thus it induces an operator on the quotient:
\begin{equation}\label{Flambda}
F(\lambda)=[T(\lambda)(A-\lambda)]: \Dom_{\max}/\Dom_{\min}\to \C^d.
\end{equation}
We denote 
\begin{equation*}
F_\Dom(\lambda)=F(\lambda)|_{\Dom/\Dom_{\min}}. 
\end{equation*}
Since $\tilde\Sing$ is isomorphic to $\Dom/\Dom_{\min}$, we conclude that $A_\Dom-\lambda$ is invertible if and only if $F_\Dom(\lambda)$ is invertible. The main properties of $F(\lambda)$ are described in Proposition~\ref{FStructure}.

On the other hand, since $B(\lambda)(A-\lambda)$ is the identity on $\Dom_{\min}$ for $\lambda\in \Lambda_R$, the operator $1-B(\lambda)(A-\lambda)$ vanishes on $\Dom_{\min}$, and induces a map 
\begin{equation}\label{Plambda}
 \big[1-B(\lambda)(A-\lambda)\big]: \Dom_{\max}/\Dom_{\min}\to x^{-m/2}L^2_b
\end{equation}
whose properties are discussed in Proposition~\ref{PStructure}. 

Finally, with the above components, the family $G_\Dom(\lambda)$ can be written as
\begin{equation}\label{GreenRemainder} 
  G_{\Dom}(\lambda) = \big[1-B(\lambda)(A-\lambda)\big]F_\Dom(\lambda)^{-1}T(\lambda).
\end{equation}

\bigskip
The proof of Theorem~\ref{ResolventStructure} relies on an analysis of associated operator families on the model cone $Y^\wedge$. These objects on $Y^\wedge$ are called wedge symbols. For instance, $A_\wedge$ is the wedge symbol of $A$. The wedge symbols of $T(\lambda)$ and $F(\lambda)$ are given by
\begin{equation*}
 T_\wedge(\lambda)=t_0(\lambda) \;\text{ and }\; F_\wedge(\lambda)= t_0(\lambda)(A_\wedge-\lambda),
\end{equation*}
where $t_0(\lambda)$ is the principal component in the expansion \eqref{TAsympExp}. 

As mentioned before, the family $G(\lambda)$ in \eqref{DminParametrix} has an asymptotic expansion in $\lambda$ similar to the one in Lemma~\ref{BSymbolExpansion}; let $G_\wedge(\lambda)$ be the principal component of that expansion. If we replace $h(x,\sigma,x^m\lambda)$ by $h(0,\sigma,x^m\lambda)$ in \eqref{Bminlambda} and denote the new operator by $Q_0(\lambda)$, then the operator family defined by
\begin{equation*}
 B_\wedge(\lambda)=Q_0(\lambda) + G_\wedge(\lambda)
\end{equation*}
is the wedge symbol of $B(\lambda)$.

Since the wedge symbols $T_\wedge(\lambda)$, $F_\wedge(\lambda)$, and $B_\wedge(\lambda)$ are related to $A_{\wedge,\Dom_\wedge}-\lambda$ in the same way how $T(\lambda)$, $F(\lambda)$, and $B(\lambda)$ are related to $A_\Dom-\lambda$, we see that $A_{\wedge,\Dom_\wedge}-\lambda$ is invertible if and only if $F_\wedge(\lambda):\Dom_{\wedge}/\Dom_{\wedge,\min}\to \C^d$ is invertible. 

\bigskip
Our main asymptotics result is the following. The details of the proof will be worked out in the next two sections.

\begin{theorem}\label{ResolventTraceExpansion}
Let $A$ and $\Lambda$ be as in Theorem~\ref{ResolventStructure}, and let $\Dom$ be stationary.
If $\ell\in\N$ is such that $m\ell>n$, then $(A_{\Dom}-\lambda)^{-\ell}$ is a smooth family of trace class operators in $x^{-m/2}L^2_b(M;E)$, and for any $\varphi \in C^{\infty}(M;\textup{End}(E))$, we have
\begin{equation*}
\Tr \bigl(\varphi(A_{\Dom}-\lambda)^{-\ell}\bigr) \sim
\sum\limits_{j=0}^{\infty}\sum\limits_{k=0}^{m_j}
\alpha_{jk}(\hat\lambda)|\lambda|^{\frac{n-j}{m}-\ell}\log^k|\lambda|
\; \text{ as } |\lambda| \to \infty,
\end{equation*}
where $\hat\lambda= \lambda/|\lambda|$ and $\alpha_{jk} \in C^{\infty}(S^1\cap\Lambda)$. The numbers $m_j$ vanish for $j < n$, and $m_n \leq 1$. In general, the $\alpha_{jk}$ depend on $\varphi$, $A$, $\Dom$, and $\ell$, but the coefficients $\alpha_{jk}$ for $j < n$ and $\alpha_{n,1}$ do not depend on $\Dom$. If both $A$ and $\varphi$ have coefficients independent of $x$ near $\partial M$, then $m_j = 0$ for all $j>n$.
\end{theorem}
\begin{remark}
The above asymptotic expansion is indeed equivalent to the expansion stated in Theorem~\ref{ResolventTraceExpansion1} due to the analyticity of the components which follows from the analyticity of the resolvent.
\end{remark}
\begin{proof}
For $\ell\in\N$ we have
\begin{align*}
\bigl(A_{\Dom}-\lambda\bigr)^{-\ell} 
 &= \frac{1}{(\ell-1)!}\partial_{\lambda}^{\,\ell-1}(A_{\Dom}-\lambda)^{-1} \\
 &=\frac{1}{(\ell-1)!}\Bigl(\partial_{\lambda}^{\,\ell-1} B(\lambda) 
 + \partial_{\lambda}^{\,\ell-1} G_{\Dom}(\lambda)\Bigr)
\end{align*}
with $B(\lambda)$ and $G_{\Dom}(\lambda)$ as in Theorem~\ref{ResolventStructure}.
Thus the statements of the theorem follow from Theorem~\ref{BCompleteExpansion} and
Corollary~\ref{TraceGStructure}.
\end{proof}

\begin{remark}
With the same arguments as for the stationary case, if $\Dom$ is nonstationary, we still obtain the partial expansion
\begin{equation*}
\Tr \bigl(\varphi(A_{\Dom}-\lambda)^{-\ell}\bigr) \sim
\sum\limits_{j=0}^{n-1} \alpha_{j}\lambda^{\frac{n-j}{m}-\ell}
+\alpha_{n}\lambda^{-\ell}\log\lambda+ O(|\lambda|^{-\ell})
\text{ as $|\lambda| \to \infty$.}
\end{equation*}
As mentioned in the introduction, the full expansion for the general case (discussed in \cite{GKM5b}) is more involved and requires a deeper understanding of $G_\Dom(\lambda)$. 
\end{remark}

\section{Asymptotic expansion of the $\Dom_{\min}$ contribution}
\label{sec:DminAsymptotics}

We start by introducing certain weighted Sobolev spaces over $M$ and $Y^\wedge$ (defined by means of the natural $L^2$ spaces introduced in Section~\ref{sec:Preliminaries}) on which the operators $A$ and $A_\wedge$ act continuously.

For a nonnegative integer $s$ we define
\begin{equation*}
 H^s_b(M;E)=\{u\in L^2_b(M;E): Pu\in L^2_b(M;E)\; \forall P\in\Diff_b^s(M;E)\}.
\end{equation*}
As usual, for a general $s\in\R$, the spaces are defined by interpolation and duality. For $\alpha\ge\beta$ and $s\ge t$, we have $x^\alpha H^s_b(M;E)\embed x^\beta H^t_b(M;E)$. If $\alpha>\beta$, this embedding is compact when $s>t$ and trace class when $s>t+n$. 

We let $H_{\rm cone}^s({Y^\wedge};E)$ be the space consisting of distributions $u$ such that given any coordinate patch $\Omega\subset Y$ diffeomorphic to an open subset of the sphere $S^{n-1}$, and given any pair of nonnegative functions $\phi \in C_c^\infty(\Omega)$ and $\omega\in C_c^\infty(\R)$ with $\omega(r)=1$ near $r=0$, we have $(1-\omega)\phi\,u \in H^s(\R^n;E)$ where $\R_+ \times S^{n-1}$ is identified with $\R^n \minus \{0\}$ via polar coordinates.

For $s,\alpha,\delta\in\R$ we define
\begin{equation*}
 \K^{s,\alpha}_\delta(Y^\wedge;E) = \omega x^{\alpha}H^s_b(Y^\wedge;E) + (1-\omega)x^{\frac{n-m}{2}-\delta} H_{\rm cone}^s({Y^\wedge};E)
\end{equation*}
for any cut-off function $\omega$. If $\delta=0$, we will omit it from the notation.
Note that 
\begin{equation*}
H_{\rm cone}^0(Y^\wedge;E) = x^{-n/2}L^2_b(Y^\wedge;E) \;\text{ and }\;
\K^{0,-m/2}(Y^\wedge;E)=x^{-m/2}L^2_b(Y^\wedge;E).
\end{equation*}
For $\alpha\ge \beta$, $\delta\ge \gamma$, and $s\ge t$, we have $\K^{s,\alpha}_\delta(Y^\wedge;E) \embed \K^{t,\beta}_{\gamma}(Y^\wedge;E)$. If $\alpha>\beta$, this embedding is compact when $s>t$ and $\delta>\gamma$, and it is trace class when $s>t+n$ and $\delta>\gamma+n$.

\bigskip
We proceed with some lemmas about the asymptotic properties of the parametrix $B(\lambda)$ in \eqref{DminParametrix}. Their proofs rely on the construction of $B(\lambda)$ combined with standard arguments from Schulze's edge theory.

\begin{lemma}\label{RapDecreasingTerms}
Let $\vp\in C^\infty(M;\textup{End}(E))$ and let $\omega_0,\,\omega_1\in C_c^\infty([0,1))$ be cut-off functions such that $\omega_0\prec \omega_1$. Then 
\begin{equation*}
 \omega_0 \vp B(\lambda)(1-\omega_1) \;\text{ and }\; (1-\omega_1)\vp B(\lambda)\omega_0
\end{equation*}
are both elements of $\S(\Lambda,\ell^1(x^{-m/2}H_b^s,x^{-m/2}H_b^t))$ for every $s,t\in\R$.
\end{lemma}

\begin{lemma}\label{BSymbolExpansion}
Let $\vp\in C^\infty(M;\textup{End}(E))$ and let $\omega,\omega_1\in C_c^\infty([0,1))$ be arbitrary cut-off functions. For $\ell\in\N$, the family $\mathcal Q(\lambda)=\vp \big(\partial_\lambda^{\ell-1}B(\lambda)\big)\omega_1(x|\lambda|^{1/m})$ has the following properties:
\begin{enumerate}[$(i)$]
\item For every $s\in\R$ and $R\gg 1$, $(1-\omega)\mathcal Q(\lambda) \in \S(\Lambda_R,\ell^1(x^{-m/2}H_b^s))$ and $\omega \mathcal Q(\lambda) \in C^\infty(\Lambda, \L(\K^{s,-m/2},\K^{s+m\ell,-m/2+\eps}_\delta))$ for all $\delta\in\R$ and some $\eps>0$; \end{enumerate}
in addition, with $q(\lambda)=\omega \mathcal Q(\lambda)$, 
\begin{enumerate}[$(i)$]
\stepcounter{enumi}
\item for every $\alpha,\beta\in\N_0$ we have
\begin{equation}\label{qSymbolEstimate}
\norm{\kappa_{|\lambda|^{1/m}}^{-1} \big(\partial_\lambda^\alpha \partial_{\bar\lambda}^\beta\, q(\lambda)\big)\kappa_{|\lambda|^{1/m}}}
= O(|\lambda|^{\frac{\mu}{m}-\alpha-\beta}) \;\text{ as }\; |\lambda|\to\infty,
\end{equation}
with $\mu=-m\ell$;
\item there are $q_j\in C^\infty(\Lambda\minus\{0\}, \L(\K^{s,-m/2},\K^{s+m\ell,-m/2+\eps}_\delta))$, $j\in\N_0$, with
\begin{equation*}
 q_j(\varrho^m\lambda)=\varrho^{-m\ell-j}
 \kappa_\varrho\, q_j(\lambda)\kappa_\varrho^{-1} \;\text{ for every } \varrho>0,
\end{equation*} 
such that for every $N\in \N$, the difference
\begin{equation*}
 q(\lambda)- \sum_{j=0}^{N-1} q_j(\lambda)
\end{equation*}
satisfies \eqref{qSymbolEstimate} with $\mu=-m\ell-N$. The leading term of the expansion is given by $q_0(\lambda)=\vp_0 \big(\partial_\lambda^{\ell-1} B_\wedge(\lambda)\big)\omega_1(x|\lambda|^{1/m})$, where $B_\wedge(\lambda)$ is the wedge symbol of $B(\lambda)$ and $\vp_0=\pi_+^*(\vp|_Y)$ with $\pi_+:Y^\wedge\to Y$.
\end{enumerate}
\end{lemma}

\begin{theorem}\label{BCompleteExpansion}
 Let $\vp\in C^\infty(M;\textup{End}(E))$ and let $B(\lambda)$ be as in \eqref{DminParametrix}. If $m\ell>n$, then $\vp\partial_{\lambda}^{\,\ell-1} B(\lambda)$ is of trace class in $x^{-m/2}L_b^2$, and 
\begin{equation*}
\Tr \big(\vp\partial_{\lambda}^{\,\ell-1} B(\lambda)\big)
\sim \sum_{j=0}^{\infty}\sum_{k=0}^{m_j}
\beta_{jk}(\hat\lambda)|\lambda|^{\frac{n-j}{m}-\ell}\log^k|\lambda|
\; \text{ as } |\lambda| \to \infty,
\end{equation*}
where $\hat\lambda= \lambda/|\lambda|$ and $\beta_{jk}\in C^\infty(S^1\cap\Lambda)$. Here $m_j = 0$ for $j < n$, and $m_j \leq 1$ for all $j\ge n$. If $A$ and $\varphi$ have coefficients independent of $x$ near $\partial M$, then
\begin{equation*}
\Tr \big(\vp\partial_{\lambda}^{\,\ell-1} B(\lambda)\big)
\sim \sum_{j=0}^{\infty}
\beta_{j,0}(\hat\lambda)|\lambda|^{\frac{n-j}{m}-\ell}
+ \beta_{n,1}(\hat\lambda)|\lambda|^{-\ell} \log|\lambda|.
\end{equation*}
\end{theorem}
\begin{proof}
Let $P(\lambda)=\vp\partial_{\lambda}^{\,\ell-1} B(\lambda)$. This operator is bounded from $x^{-m/2}L_b^2(M;E)$ to $x^{-m/2+\eps}H_b^{m\ell}(M;E)$ for some $\eps>0$, so it is trace class in $x^{-m/2}L_b^2(M;E)$ since the embedding $x^{-m/2+\eps}H_b^{m\ell}\embed x^{-m/2}L_b^2$ is trace class when $m\ell>\dim M$.

Choose cut-off functions $\omega,\omega_0, \omega_1\in C_c^\infty([0,1))$ such that  $\omega_0\prec \omega\prec \omega_1$, and rewrite
\begin{align*}
 P(\lambda) &= \omega P(\lambda) + (1-\omega)P(\lambda) \\
 &=\omega P(\lambda)\omega_1 + \omega P(\lambda)(1-\omega_1) + 
 (1-\omega)P(\lambda)\omega_0 + (1-\omega)P(\lambda)(1-\omega_0).
\end{align*}
Then, by Lemma~\ref{RapDecreasingTerms},
\begin{equation*}
 P(\lambda)\equiv \omega P(\lambda)\omega_1 + (1-\omega)P(\lambda)(1-\omega_0)
\end{equation*}
modulo an element in $\bigcap_{s,t\in\R} \S(\Lambda,\ell^1(x^{-m/2}H_b^s,x^{-m/2}H_b^t))$. 

Let $P_{\rm int}(\lambda)=(1-\omega)P(\lambda)(1-\omega_0)$. This is a standard parameter-dependent family of trace class operators over the interior of $M$, and it is well-known that
\begin{equation*}
 \Tr P_{\rm int}(\lambda) \sim \sum_{j=0}^\infty \beta_j(\hat\lambda)|\lambda|^{\frac{n-j}{m}-\ell}
\end{equation*}
with coefficients $\beta_{j}\in C^\infty(S^1\cap\Lambda)$, see e.g. \cite{Gilkey,GruSee95}. 

Let $\varrho=|\lambda|^{1/m}\ge 1$ and split
\begin{equation}\label{BoundaryPlambda}
 \omega P(\lambda)\omega_1 = \omega P(\lambda)(1-\omega(x\varrho))\omega_1 
 + \omega P(\lambda)\omega(x\varrho)\omega_1.
\end{equation}
Let $q(\lambda)=\omega P(\lambda)\omega(x\varrho)\omega_1=\omega P(\lambda)\omega(x\varrho)$. For $N\in\N$ we use Lemma~\ref{BSymbolExpansion} to write
\begin{equation*}
 q(\lambda)= \sum_{j=0}^{N-1} q_j(\lambda) + q_{[N]}(\lambda).
\end{equation*}
Since $m\ell>n$, all components of $q(\lambda)$ are trace class, and since the remainder $q_{[N]}(\lambda)$ satisfies \eqref{qSymbolEstimate} with $\alpha=\beta=0$ and $\mu=-m\ell-N$ for every $s\in\R$, we get
\begin{equation*}
 \Tr q_{[N]}(\lambda) = \Tr \big(\kappa_{\varrho}^{-1}q_{[N]}(\lambda)\kappa_{\varrho}\big) = O(|\lambda|^{-\frac{N}{m}-\ell})\;\text{ as }\; |\lambda|\to\infty.
\end{equation*}

On the other hand, since $q_j(\lambda)= q_j(\varrho^m\hat\lambda)= \varrho^{-m\ell-j}\kappa_{\varrho}q_j(\hat\lambda) \kappa_{\varrho}^{-1}$, we get
\begin{equation*}
\Tr q_j(\lambda)= \varrho^{-m\ell-j} \Tr\big(\kappa_{\varrho}q_j(\hat\lambda) \kappa_{\varrho}^{-1}\big) = \varrho^{-m\ell-j}\Tr q_j(\hat\lambda),
\end{equation*}
and thus
\begin{equation*}
 \Tr\big(\omega P(\lambda)\omega(x\varrho)\big) = \Tr q(\lambda) = \sum_{j=0}^{N-1} \beta_j'(\hat\lambda)|\lambda|^{-\frac{j}{m}-\ell} + O(|\lambda|^{-\frac{N}{m}-\ell})
\end{equation*}
with $\beta_j'(\hat\lambda)=\Tr q_j(\hat\lambda)$.

In view of \eqref{Bminlambda}, if we choose $\omega\prec\tilde\omega$ and $\omega_1\prec \tilde\omega_1$, then the first component of the right-hand side of equation \eqref{BoundaryPlambda} becomes
\begin{equation*}
 \omega P(\lambda)(1-\omega(x|\lambda|^{1/m}))\omega_1
 = P_{\log}(\lambda) + \omega g(\lambda) \omega_1,
\end{equation*}
where
\begin{equation*}
 P_{\log}(\lambda) =\omega \vp \big(\partial_{\lambda}^{\,\ell-1} Q(\lambda)\big) (1-\omega(x|\lambda|^{1/m}))\omega_1
\end{equation*}
and $g(\lambda)$ is a Green remainder with an expansion in $\lambda$ similar to the one for $q(\lambda)$. In fact,
\begin{equation*}
 \Tr\big(\omega g(\lambda) \omega_1\big)\sim \sum_{j=0}^\infty \beta_j''(\hat\lambda)|\lambda|^{-\frac{j}{m}-\ell}
\end{equation*}
with coefficients $\beta_{j}''\in C^\infty(S^1\cap\Lambda)$. 

It remains to expand $\Tr P_{\log}(\lambda)$. First of all, observe that the family $\vp \partial_{\lambda}^{\,\ell-1} Q(\lambda)$ is of the form \eqref{Bminlambda} with $x^m h(x,\sigma,x^m\lambda)$ replaced by $x^{m\ell} h^{(\ell)}(x,\sigma,x^m\lambda)$, where
\begin{equation*}
h^{(\ell)}(x,\sigma,\lambda) =\vp(x)\big(\partial_{\lambda}^{\,\ell-1} h\big)(x,\sigma,\lambda) \in C^\infty(\overline{\R}_+,L_{c\ell}^{-m\ell,(1,m)}(Y;\R\times\Lambda)).
\end{equation*}
By means of a Taylor expansion at $x=0$, we can write
\begin{equation*}
 h^{(\ell)}(x,\sigma,\lambda)=\sum_{j=0}^{N-1} x^j h_{j}(\sigma,\lambda)
 + x^N h_{[N]}(x,\sigma,\lambda)
\end{equation*}
and obtain a decomposition
\begin{equation*}
 P_{\log}(\lambda) =\omega\Big(\sum_{j=0}^{N-1}Q_j(\lambda)+Q_{[N]}(\lambda)\Big) (1-\omega(x|\lambda|^{1/m}))\omega_1, 
\end{equation*}
where $Q_j(\lambda)$ and $Q_{[N]}(\lambda)$ are of the form \eqref{Bminlambda} with $x^m h(x,\sigma,x^m\lambda)$ replaced by $x^{m\ell+j}h_j(\sigma,x^m\lambda)$ and $x^{m\ell+N}h_{[N]}(x,\sigma,x^m\lambda)$, respectively. This induces a decomposition of the trace
\begin{equation}\label{PlogTrace}
 \Tr P_{\log}(\lambda) = \sum_{j=0}^{N-1} \tau_j(\lambda) + \tau_{[N]}(\lambda)
\end{equation}
with the obvious meaning of notation. For every $j<N$,
\begin{equation}\label{taujLambda}
 \tau_j(\lambda)= \int_Y\int_0^\infty x^{m\ell+j} \omega(x) (1-\omega(x|\lambda|^{1/m})) k_{j}(y,x^m\lambda)\,\frac{dx}{x}dy,
\end{equation}
where $k_{j}(y,\lambda)$ is locally given by
\begin{equation}\label{LocalKernel}
 k_{j}(y,\lambda)=\int_{\R^n} p(y,\eta,\lambda) \,\dbar\eta
\end{equation}
for some parameter-dependent classical symbol $p(y,\eta,\lambda)$ of order $-m\ell$. Here we use the notation $\eta=(\sigma,\xi)\in \R\times\R^{n-1}$ and $\dbar\eta= \frac{1}{(2\pi)^{n}}d\eta$.

To simplify the notation, and without loss of generality, we assume $\omega(x)=1$ for $0\le x\le 1$ and $\omega(x)=0$ for $x\ge 2$. In particular, $1-\omega(x\varrho) = 0$ for $0\le x\le \tfrac{1}{\varrho}$. 

Fix $j$ and let 
\begin{align*}
 s(y,\lambda)&= \int_0^\infty x^{m\ell+j} \omega(x) (1-\omega(x\varrho)) k_{j}(y,x^m\lambda)\,\frac{dx}{x}\\
 &= \int_{1/\varrho}^\infty x^{m\ell+j} \omega(x) (1-\omega(x\varrho)) k_{j}(y,x^m\lambda)\,\frac{dx}{x}.
\end{align*}
For $J\in \N$ we expand 
\begin{equation*}
p(y,\eta,\lambda) = \sum_{k=0}^{J-1} p_k(y,\eta,\lambda) + p_{[J]}(y,\eta,\lambda),
\end{equation*}
where $p_{[J]}$ is a symbol of order $-m\ell-J$, and for $|\eta|+|\lambda|^{1/m}\ge 1$,
\begin{equation*}
p_k(y,t\eta,t^m\lambda)=t^{-m\ell-k}p_k(y,t\eta,t^m\lambda) \text{ for every } t\ge 1.
\end{equation*}
By \eqref{LocalKernel} this expansion induces a decomposition $s(y,\lambda)=s_0(y,\lambda)+\cdots+s_{[J]}(y,\lambda)$. We will obtain an expansion in $\lambda$ of $\tau_j(\lambda)$ through an expansion in $\lambda$ of every component of $s(y,\lambda)$. Most of the computations are done locally over a patch $\Omega$ of $Y$ and put together by means of a partition of unity. 

Let $\hat\lambda=\frac{\lambda}{|\lambda|}$ and let $k_{[J]}(y,\lambda)$ be the function locally defined by $\int p_{[J]}(y,\eta,\lambda) \,\dbar\eta$. Then
\begin{align*}
 s_{[J]}(y,\lambda)&= \int_{1/\varrho}^\infty x^{m\ell+j} \omega(x) (1-\omega(x\varrho)) k_{[J]}(y,(x\varrho)^m\hat \lambda)\,\frac{dx}{x}\\
 &= \varrho^{-m\ell-j}\int_{1}^\infty x^{m\ell+j} \omega(x/\varrho) (1-\omega(x)) k_{[J]}(y,x^m\hat \lambda)\,\frac{dx}{x}.
\end{align*}
If we write $\omega(x)=1+x^{N-j}\omega_{[N-j]}(x)$ and choose $J=N+1$, then
\begin{align*}
 s_{[J]}(y,\lambda) &= \varrho^{-m\ell-j}\int_{1}^\infty x^{m\ell+j} (1-\omega(x)) k_{[J]}(y,x^m\hat \lambda)\,\frac{dx}{x} \\
 &\quad + \varrho^{-m\ell-N}\int_{1}^\infty x^{m\ell+N} \omega_{[N-j]}(x/\varrho)(1-\omega(x)) k_{[J]}(y,x^m\hat \lambda)\,\frac{dx}{x},
\end{align*}
and the last integral converges uniformly in $\varrho$. Thus
\begin{equation}\label{sJExpansion}
 s_{[J]}(y,\lambda) = \alpha(y,\hat\lambda) |\lambda|^{-\frac{j}{m}-\ell} +O(|\lambda|^{-\frac{N}{m}-\ell})
\end{equation}
with $\alpha(y,\hat\lambda)$ depending on $\ell,j$, and $J$.

We now proceed to expand $s_k(y,\lambda)$ for $0\le k< J$. We assume $\varrho\ge 2$ and write
\begin{equation*}
 s_{k}(y,\lambda) = \int_{1/\varrho}^\infty x^{m\ell+j} \omega(x)(1-\omega(x\varrho)) \left(\int_{\R^n} p_k(y,\eta,(x\varrho)^m\hat\lambda)\,\dbar\eta\right)\frac{dx}{x}.
\end{equation*}
The change $\eta\to (x\varrho)\eta$ and the homogeneity of $p_k$ (with $t=x\varrho\ge 1$) give
\begin{align*}
 s_{k}(y,\lambda) &= \varrho^{-m\ell-k+n}\int_{1/\varrho}^\infty x^{j-k+n} \omega(x)(1-\omega(x\varrho)) \left(\int p_k(y,\eta,\hat\lambda) \,\dbar\eta\right)\frac{dx}{x} \\
 &= \varrho^{-m\ell-k+n} \left(\int p_k(y,\eta,\hat\lambda) \,\dbar\eta\right) \int_{1/\varrho}^\infty x^{j-k+n} \omega(x)(1-\omega(x\varrho))\frac{dx}{x}.
\end{align*}
Since
\begin{equation*} 
 \int_{1/\varrho}^\infty x^{j-k+n} \omega(x)(1-\omega(x\varrho))\frac{dx}{x}
 = \begin{cases} 
     c_{1}+c_{2}\,\varrho^{-j+k-n} & \text{if } k\not=j+n,\\ 
     \log\varrho & \text{if } k=j+n, 
   \end{cases}
\end{equation*}
for some constants $c_1,c_2\in\R$, we get
\begin{equation}\label{skExpansion}
 s_{k}(y,\lambda)= \alpha(y,\hat\lambda) |\lambda|^{\frac{n-k}{m}-\ell}
 + \alpha'(y,\hat\lambda) |\lambda|^{-\frac{j}{m}-\ell} 
 + \alpha''(y,\hat\lambda) |\lambda|^{-\frac{j}{m}-\ell}\log|\lambda|
\end{equation}
with functions $\alpha, \alpha', \alpha''\in C^\infty(\Omega\times(S^1\cap\Lambda))$ that depend on $\ell,j$, and $k$. In particular, $\alpha(y,\hat\lambda)=0$ for $k<n$, and
\begin{equation*}
 \alpha''(y,\hat\lambda) =
 \begin{cases}
 \frac1m \int_{\R^n} p_k(y,\eta,\hat\lambda)\,\dbar\eta& \text{if } k=j+n, \\
 \qquad\qquad 0 &\text{if } k\not=j+n.
 \end{cases}
\end{equation*}

Similar to \eqref{taujLambda}, we have
\begin{equation*}
 \tau_{[N]}(\lambda)= \int_Y\int_0^\infty x^{m\ell+N} \omega(x) (1-\omega(x\varrho)) k_{N}(x,y,x^m\lambda)\,\frac{dx}{x}dy,
\end{equation*}
where $k_{N}(x,y,\lambda)$ is locally given by
\begin{equation*}
 k_{N}(x,y,\lambda)=\int_{\R^n} q(x,y,\eta,\lambda) \,\dbar\eta
\end{equation*}
for some parameter-dependent classical symbol $q(x,y,\eta,\lambda)$ of order $-m\ell$. Let
\begin{equation*}
 t_N(y,\lambda)=\int_0^\infty x^{m\ell+N} \omega(x) (1-\omega(x\varrho)) 
 \left(\int q(x,y,\eta,x^m\lambda) \,\dbar\eta\right)\frac{dx}{x}.
\end{equation*}
We let $J=N+1$ and expand $q$ in homogeneous components with a remainder of order $-m\ell-J$. This gives a decomposition $t_N(y,\lambda)=t_{N,0}(y,\lambda)+\cdots+t_{N,[J]}(y,\lambda)$.
First, we examine $t_{N,[J]}$. With the change of variables $x\to x/\varrho$, we get 
\begin{equation*}
t_{N,[J]}(y,\lambda)=\varrho^{-m\ell-N}\!\int_0^\infty x^{m\ell+N} (1-\omega(x)) \left(\int \omega(\tfrac{x}{\varrho})q_{[J]}(\tfrac{x}{\varrho},y,\eta,x^m\hat\lambda) \,\dbar\eta\right)\frac{dx}{x}
\end{equation*}
and the integral converges uniformly in $\varrho$. Thus
\begin{equation}\label{sNJExpansion}
 t_{N,[J]}(y,\lambda) = O(|\lambda|^{-\frac{N}{m}-\ell}) 
 \;\text{ as } |\lambda|\to\infty.
\end{equation}

For $0\le k<J$ we have
\begin{align*}
 t_{N,k}(y,\lambda) 
 &= \int_0^\infty x^{m\ell+N} \omega(x) (1-\omega(x\varrho)) 
 \left(\int_{\R^n} q_k(x,y,\eta,(x\varrho)^m\hat\lambda) \,\dbar\eta\right)\frac{dx}{x}\\
 &= \int_0^\infty x^{m\ell+N} \omega(x) (1-\omega(x\varrho)) (x\varrho)^{-m\ell-k+n}
 \left(\int_{\R^n} q_k(x,y,\eta,\hat\lambda) \,\dbar\eta\right)\frac{dx}{x}
\end{align*}
using the change of variables $\eta\to (x\varrho)\eta$ and the homogeneity of $q_k$. Thus, 
\begin{equation*}
\begin{split} 
 t_{N,k}(y,\lambda) 
 &= \varrho^{-m\ell-k+n}\int_0^\infty x^{N-k+n} \omega(x) \left(\int q_k(x,y,\eta,\hat\lambda) \,\dbar\eta\right) \frac{dx}{x} \\ 
 &\qquad -\int_0^\infty x^{m\ell+N} \omega(x)\omega(x\varrho) (x\varrho)^{-m\ell-k+n}
 \left(\int q_k(x,y,\eta,\hat\lambda) \,\dbar\eta\right)\frac{dx}{x}.
\end{split}
\end{equation*}
Now, with the change of variables $x\to x/\varrho$, the last integral becomes
\begin{equation*}
 \varrho^{-m\ell-N}\!\int_0^\infty x^{N-k+n} \omega(x) \left(\int \omega(x/\varrho) q_k(x/\varrho,y,\eta,\hat\lambda) \,\dbar\eta\right)\frac{dx}{x},
\end{equation*}
and the integral converges uniformly in $\varrho$. In conclusion, 
\begin{equation}\label{sNkExpansion}
 t_{N,k}(y,\lambda)= \alpha(y,\hat\lambda)|\lambda|^{\frac{n-k}{m}-\ell} + O(|\lambda|^{-\frac{N}{m}-\ell})
\end{equation}
with $\alpha(y,\hat\lambda)=\int_0^\infty x^{N-k+n} \omega(x) \left(\int q_k(x,y,\eta,\hat\lambda) \,\dbar\eta\right) \frac{dx}{x}$.

Finally, integrating over $Y$ the expressions in \eqref{sJExpansion}, \eqref{skExpansion}, \eqref{sNJExpansion}, and \eqref{sNkExpansion}, we arrive at the expansion
\begin{equation*}
 \Tr P_{\log}(\lambda) \sim \sum_{j=0}^\infty
 \gamma_j(\hat\lambda) |\lambda|^{\frac{n-j}{m}-\ell}
 + \sum_{j=0}^\infty \gamma_j'(\hat\lambda) |\lambda|^{-\frac{j}{m}-\ell}\log|\lambda|
\end{equation*}
with coefficients $\gamma_j, \gamma_j'\in C^\infty(S^1\cap\Lambda)$.

If $A$ and $\varphi$ have coefficients independent of $x$ near $\partial M$, then so does the Mellin symbol $h^{(\ell)}(x,\sigma,\lambda)$ of $\vp\partial_{\lambda}^{\,\ell-1} B(\lambda)$ and there is no need for a Taylor expansion. In other words, $\Tr P_{\log}(\lambda)=\tau_0(\lambda)$ in \eqref{PlogTrace}. In this case, \eqref{skExpansion} becomes
\begin{equation*}
 s_{k}(y,\lambda)= \alpha(y,\hat\lambda) |\lambda|^{\frac{n-k}{m}-\ell}
 + \alpha'(y,\hat\lambda) |\lambda|^{-\ell} 
 + \alpha''(y,\hat\lambda) |\lambda|^{-\ell}\log|\lambda|
\end{equation*}
with $\alpha(y,\hat\lambda)=0$ for $k<n$, and
\begin{equation*}
 \alpha''(y,\hat\lambda) =
 \begin{cases}
 \frac1m \int_{\R^n} p_k(y,\eta,\hat\lambda)\,\dbar\eta& \text{if } k=n, \\
 \qquad\qquad 0 &\text{if } k\not=n.
 \end{cases}
\end{equation*}
Consequently, there is only one $\log$ term in the expansion of $\Tr P(\lambda)$.
\end{proof}

\section{Asymptotic expansion of the finite rank contribution}
\label{sec:FiniteContribution}

We let $A$ be a cone differential operator of order $m$ that satisfies all the ellipticity conditions outlined at the beginning of Section~\ref{sec:ResolventStructure}. In addition, let $\Dom$ be stationary. 

In this section, we discuss the structure and asymptotic properties of the family $G_\Dom(\lambda)$ described in \eqref{GreenRemainder}. Our analysis leads to a full asymptotic expansion of $\Tr G_\Dom(\lambda)$ as $|\lambda|\to\infty$. For simplicity, let $L^2_b=L^2_b(M;E)$ and $L^2_{b,\wedge}=L^2_{b}(Y^\wedge;E)$.

\begin{theorem}\label{GStructure}
Let $\vp\in C^\infty(M;\textup{End}(E))$ and let $R>0$ be such that $G_\Dom(\lambda)$ exists for every $\lambda\in\Lambda_R$. Let $\omega,\tilde\omega\in C_c^\infty([0,1))$ be arbitrary cut-off functions. Then:
\begin{enumerate}[$(i)$]
\item The families $(1-\omega)\vp G_\Dom(\lambda)$ and $\vp G_\Dom(\lambda)(1-\omega)$ are in $\S(\Lambda_R,\ell^1(x^{-m/2}L^2_b))$, and $\omega\vp G_\Dom(\lambda)\tilde \omega \in C^\infty(\Lambda_R,\ell^1(x^{-m/2}L^2_{b,\wedge}))$;
\end{enumerate}
in addition, with $g(\lambda)=\omega\vp G_\Dom(\lambda)\tilde \omega$, 
\begin{enumerate}[$(i)$] \stepcounter{enumi}
\item for every $\alpha,\beta\in\N_0$ we have
\begin{equation}\label{gSymbolEstimate}
\norm{\kappa_{|\lambda|^{1/m}}^{-1}\big(\partial_\lambda^\alpha \partial_{\bar\lambda}^\beta\, g(\lambda)\big)\kappa_{|\lambda|^{1/m}}}_{\ell^1} = O(|\lambda|^{\frac{\mu}{m}-\alpha-\beta}) \;\text{ as }\; |\lambda|\to\infty,
\end{equation}
with $\mu=-m$; 
\item for every $j\in\N_0$ there exist $m_j\in\N_0$ and $g_{jk}\in C^\infty(\Lambda\minus\{0\},\ell^1(x^{-m/2}L^2_{b,\wedge}))$, $k=0,\dots,m_j$, with
\begin{equation*}
 g_{jk}(\varrho^m\lambda)=\varrho^{-m-j}\kappa_\varrho
 g_{jk}(\lambda)\kappa_{\varrho}^{-1} \;\text{ for every } \varrho>0,
\end{equation*} 
such that for every $N\in \N$, the difference
\begin{equation}\label{gAsympExp}
 g(\lambda)- \sum_{j=0}^{N-1}\sum_{k=0}^{m_j} g_{jk}(\lambda) \log^k|\lambda| 
\end{equation}
satisfies \eqref{gSymbolEstimate} with $\mu=-m-N+\eps$ for any $\eps>0$. 
Here $m_0=0$, and 
\[ g_{00}(\lambda)= \vp_0[1-B_\wedge(\lambda)(A_\wedge-\lambda)]\theta\, \phi_{00}(\lambda) t_0(\lambda)\] 
with $t_0(\lambda)$ as in \eqref{TAsympExp} and $\phi_{00}(\lambda)$ as in \eqref{FInvAsympExp}. If $A$ has coefficients independent of $x$ near $\partial M$, then $m_j=0$ for all $j$.
\end{enumerate}
\end{theorem}

\begin{corollary}\label{TraceGStructure}
For $R> 0$ sufficiently large and $\vp\in C^\infty(M;\textup{End}(E))$, the operator family $\varphi G_D(\lambda)$ is a smooth family of trace class operators in $x^{-m/2}L^2_b$ for $\lambda \in \Lambda_R$. We have an asymptotic expansion
\begin{equation*}
\Tr \bigl(\varphi G_D(\lambda)\bigr) \sim
\sum\limits_{j=0}^{\infty}\sum\limits_{k=0}^{m_j}\gamma_{jk}(\hat{\lambda}) |\lambda|^{-\frac{j}{m}-1}\log^k|\lambda| \; \text{ as } |\lambda| \to \infty,
\end{equation*}
where $\hat\lambda= \lambda/|\lambda|$, $\gamma_{jk} \in C^{\infty}(S^1\cap\Lambda)$, and $m_0 =0$. This expansion can be differentiated formally to obtain expansions of $\Tr \bigl(\varphi
\partial^{\alpha}_{\lambda}\partial^{\beta}_{\bar\lambda}G_D(\lambda)\bigr)$ for any $\alpha,\beta \in \N_0$. If $A$ has coefficients independent of $x$ near $\partial M$, then $m_j=0$ for all $j$.
\end{corollary}
\begin{proof}
For cut-off functions $\omega,\tilde{\omega} \in C_c^{\infty}([0,1))$ write
\begin{equation*}
\varphi G_D(\lambda) = \omega \varphi G_D(\lambda) \tilde{\omega} +
(1-\omega)\varphi G_D(\lambda) + \omega \varphi G_D(\lambda)(1-\tilde{\omega}).
\end{equation*}
Both $(1-\omega)\varphi G_D(\lambda)$ and $\omega \varphi G_D(\lambda)(1-\tilde{\omega})$ are smooth and rapidly decreasing (with all derivatives) taking values in the trace class
operators. Hence they are both negligible. Moreover, $g(\lambda) = \omega \varphi G_D(\lambda) \tilde{\omega}$ is smooth taking values in the trace class operators, and we have
\begin{equation*}
\Tr g(\lambda) \sim \sum\limits_{j=0}^{\infty}\sum\limits_{k=0}^{m_j}
\Tr(g_{jk}(\lambda))\log^k|\lambda| \; \text{ as } |\lambda| \to \infty,
\end{equation*}
with the $g_{jk}(\lambda)$ as in \eqref{gAsympExp}. Since
\begin{equation*}
g_{jk}(\lambda) = g_{jk}(|\lambda|\hat{\lambda}) = |\lambda|^{-\frac{j}{m}-1}\kappa_{|\lambda|^{1/m}} g_{jk}(\hat{\lambda})
\kappa_{|\lambda|^{1/m}}^{-1} \;\text{ for } |\lambda|>0, 
\end{equation*}
we get $\Tr(g_{jk}(\lambda)) = \Tr (g_{jk}(\hat{\lambda})) |\lambda|^{-\frac{j}{m}-1}$. We let $\gamma_{jk}(\hat{\lambda})= \Tr (g_{jk}(\hat{\lambda}))$.

It is clear that the above asymptotic expansion can be differentiated formally. In fact, by Theorem~\ref{GStructure} and because of the trace property, we have that $\Tr(\varphi G_D(\lambda))$ is a scalar $\log$-polyhomogeneous symbol in the sector $\Lambda_R$. The claimed asymptotic expansion is thus an asymptotic expansion of symbols in $\Lambda_R$.
\end{proof}

The proof of Theorem~\ref{GStructure} follows from structural results that we will present in the next set of propositions. We let $R>0$ be as in Theorem~\ref{ResolventStructure}.

The first proposition is a direct consequence of the parametrix construction for $A_\Dom-\lambda$ given in \cite[Section~5]{GKM2}.

\begin{proposition}\label{Tlambda}
Let $\omega\in C_c^\infty([0,1))$ be an arbitrary cut-off function.
The operator family $T(\lambda)$ has the following properties:
\begin{enumerate}[$(i)$]
\item For every $s\in\R$ we have $T(\lambda)(1-\omega) \in\S(\Lambda_R,\L(x^{-m/2}H_b^s,\C^d))$ and $T(\lambda)\omega \in C^\infty(\Lambda_R,\L(\K^{s,-m/2},\C^d))$; 
\end{enumerate}
in addition, with $t(\lambda)=T(\lambda)\omega$, 
\begin{enumerate}[$(i)$] \stepcounter{enumi}
\item for every $\alpha,\beta\in\N_0$ we have
\begin{equation}\label{TSymbolEstimate}
\norm{\big(\partial_\lambda^\alpha \partial_{\bar\lambda}^\beta\, t(\lambda)\big)\kappa_{|\lambda|^{1/m}}}
= O(|\lambda|^{\frac{\mu}{m}-\alpha-\beta}) \;\text{ as }\; |\lambda|\to\infty,
\end{equation}
with $\mu=-m$;
\item there are $t_j\in C^\infty(\Lambda\minus\{0\},\L(\K^{s,-m/2},\C^d))$, $j\in\N_0$, with
\begin{equation*}
 t_j(\varrho^m\lambda)=\varrho^{-m-j}
 t_j(\lambda)\kappa_\varrho^{-1} \;\text{ for every } \varrho>0,
\end{equation*} 
such that for every $N\in \N$, the difference
\begin{equation}\label{TAsympExp}
 t(\lambda)- \sum_{j=0}^{N-1} t_j(\lambda)
\end{equation}
satisfies \eqref{TSymbolEstimate} with $\mu=-m-N$.
\end{enumerate}
\end{proposition}

Before we discuss the properties of the family $F(\lambda)$ introduced in \eqref{Flambda}, we need to make some identifications and introduce some notation.

Recall that $\Sigma$ is the subset of the boundary spectrum given by 
\[  \Sigma=\spec_b(A)\cap \{\sigma\in\C: -m/2<\Im\sigma<m/2\}. \]
For every $\sigma_0\in \Sigma$, the space $\Sing_{\wedge,\sigma_0}$ consists of singular functions of the form 
\begin{equation*}
 x^{i\sigma_0}\sum_{k=0}^{\mu_{\sigma_0}} c_{\sigma_0,k}(y)\log^k x 
 \;\text{ with } c_{\sigma_0,k}\in C^\infty(Y;E),
\end{equation*}
such that
\begin{equation*}
 \Dom_{\wedge,\max}/\Dom_{\wedge,\min} \cong \Sing_{\wedge,\max}= \bigoplus_{\sigma_0\in\Sigma} \Sing_{\wedge,\sigma_0}\subset C^\infty(\open Y^\wedge;E).
\end{equation*}
This identification is given by the map
\begin{equation*}
 u\mapsto (\omega u+\Dom_{\wedge,\min}): \Sing_{\wedge,\max}\to \Dom_{\wedge,\max}/\Dom_{\wedge,\min}
\end{equation*}
for an arbitrary cut-off function $\omega\in C_c^\infty([0,1))$. Without change of notation, we will identify maps on/to $ \Dom_{\wedge,\max}/\Dom_{\wedge,\min}$ with maps on/to $\Sing_{\wedge,\max}$.

For $\ell\in\N_0$ and $u\in \Sing_{\wedge,\sigma_0}$ let $\theta_{\ell}^{-1}$ be defined by
\begin{equation*}
  \theta_{\ell}^{-1}u=u+\sum_{k\in J_{\sigma_0,\ell}} \e_{\sigma_0,k}\,u,
\end{equation*}
where 
\begin{equation*}
\begin{gathered} 
 J_{\sigma_0,\ell}=\{k\in\N: \Im\sigma_0-k\ge -m/2-\ell\}, \\
\end{gathered}
\end{equation*}
and the $\e_{\sigma_0,k}$ are the operators defined in \eqref{eSigmak}. 
The map $\theta_{\ell}^{-1}$ extends to $\Sing_{\wedge,\max}$ in the obvious way. If we define
\begin{equation*}
 \Sing_{\max}^{(\ell)} = \theta_{\ell}^{-1}\big(\Sing_{\wedge,\max}\big) 
 \subset C^\infty(\open Y^\wedge;E),
\end{equation*}
we can then identify $\Dom_{\max}/\Dom_{\min}$ with $\Sing_{\max}^{(\ell)}$ via the commutative diagram
\begin{center}
\begin{picture}(80,70)
\put(0,55){$\Sing_{\max}^{(\ell)} \cong \Dom_{\max}/\Dom_{\min}$}
\put(10,18){\vector(0,1){30}}
\put(60,18){\vector(0,1){30}}
\put(-6,30){\footnotesize $\theta_{\ell}^{-1}$}
\put(64,30){\footnotesize $\theta^{-1}$}
\put(-8,5){$\Sing_{\wedge,\max} \cong \Dom_{\wedge,\max}/\Dom_{\wedge,\min}$}
\end{picture}
\end{center}
which gives the map
\begin{equation*}
 u\mapsto (\omega u+\Dom_{\min}): \Sing_{\max}^{(\ell)}\to \Dom_{\max}/\Dom_{\min}
\end{equation*}
for an arbitrary cut-off function $\omega\in C_c^\infty([0,1))$.

For $\sigma_0\in\Sigma$, $\vt\in\N$, and $\varrho>0$ let 
\begin{equation*}
 \e_{\sigma_0,\vt}(\varrho)=\varrho^{\vt}\kappa_{\varrho}^{-1}\e_{\sigma_0,\vt}\,\kappa_{\varrho}: \Sing_{\wedge,\sigma_0}\to C^\infty(\open Y^\wedge;E),
\end{equation*}
and let 
\begin{equation*}
 \tilde\Sing_{\sigma_0}= \theta^{-1}\big(\Sing_{\wedge,\sigma_0}\big)
 \subset \Dom_{\max}/\Dom_{\min}.
\end{equation*}
For $u\in \tilde\Sing_{\sigma_0}$ and a fixed cut-off function $\omega_0$, let
\begin{equation}\label{Lrho&Ktilde}
 L_\varrho^{(\ell)} u = \theta u + \sum_{\vt\in J_{\sigma_0,\ell}} \varrho^{-\vt} \e_{\sigma_0,\vt}(\varrho)(\theta u)
 \quad\text{and}\quad \Ka_\ell(\varrho)u=\kappa_\varrho(\omega_0 L_\varrho^{(\ell)}u).
\end{equation}
These operators extend to 
\begin{equation*}
 \Dom_{\max}/\Dom_{\min}=\bigoplus_{\sigma_0\in\Sigma} \tilde\Sing_{\sigma_0}
\end{equation*}
and take values in $C^\infty(\open Y^\wedge;E)$. However, for $\varrho\ge 1$, we will rather consider
\begin{equation*}
 \Ka_{\ell}(\varrho): \Dom_{\max}/\Dom_{\min}\to \Dom_{\max}. 
\end{equation*}
This is possible because of $\omega_0$. Moreover, observe that
\begin{align*}
 \kappa_\varrho L_\varrho^{(\ell)}u &= \kappa_\varrho\theta u + \sum_{\vt\in J_{\sigma_0,\ell}} \varrho^{-\vt} \kappa_\varrho \e_{\sigma_0,\vt}(\varrho)(\theta u) \\
 &= \kappa_\varrho\theta u + \sum_{\vt\in J_{\sigma_0,\ell}} \e_{\sigma_0,\vt}\, \kappa_\varrho\theta u\\
 &= \theta_{\ell}^{-1} (\kappa_\varrho\theta u).
\end{align*}
The map $\Ka_{\ell}(\varrho)$ is a lift of the action
\begin{equation}\label{kappatilde}
 \tilde \kappa_\varrho=\theta^{-1} \kappa_\varrho\theta:
 \Dom_{\max}/\Dom_{\min}\to \Dom_{\max}/\Dom_{\min}
\end{equation}
in the sense that the diagram

\begin{center}
\begin{picture}(100,95)
\put(100,85){$\Dom_{\max}$} 
\put(112,80){\vector(0,-1){30}}
\put(115,60){\footnotesize $q$}
\put(25,70){\footnotesize $\Ka_\ell(\varrho)$}
\put(10,48){\curve(0,0,82,32)}
\put(92,80){\vector(3,1){1}}
\put(-25,35){$\Dom_{\max}/\Dom_{\min}$}
\put(27,37){\vector(1,0){27}}
\put(25,25){\footnotesize $\kappa_\varrho L^{(\ell)}_\varrho$}
\put(60,35){$\Sing_{\max}^{(\ell)}\cong \Dom_{\max}/\Dom_{\min}$}
\put(-5,-5){\curve(10,30,55,23,100,30)}
\put(45,10){\footnotesize $\tilde\kappa_\varrho$}
\put(96,26){\vector(2,1){1}}
\end{picture}
\end{center}
commutes, where $q:\Dom_{\max}\to \Dom_{\max}/\Dom_{\min}$ is the quotient map.

\begin{proposition}\label{FStructure}
The family $F(\lambda)=[T(\lambda)(A-\lambda)]$ has the following properties:
\begin{enumerate}[$(i)$]
\item $F(\lambda) \in C^\infty(\Lambda_R,\L(\Dom_{\max}/\Dom_{\min},\C^d))$ and for every $\alpha,\beta\in\N_0$ we have
\begin{equation}\label{FSymbolEstimate}
\norm{\big(\partial_\lambda^\alpha \partial_{\bar\lambda}^\beta\, F(\lambda)\big)\tilde\kappa_{|\lambda|^{1/m}}} = O(|\lambda|^{\frac{\mu}{m}-\alpha-\beta}) \;\text{ as }\; |\lambda|\to\infty,
\end{equation}
with $\mu=0$;
\item for all $j\in\N_0$ there exist $n_j\in\N_0$ and $f_{jk}\in C^\infty(\Lambda\minus\{0\},\L(\Dom_{\max}/\Dom_{\min},\C^d))$, $k=0,\dots,n_j$, with
\begin{equation*}
 f_{jk}(\varrho^m\lambda)=\varrho^{-j}
 f_{jk}(\lambda)\tilde\kappa_\varrho^{-1} \;\text{ for every } \varrho>0,
\end{equation*} 
such that for every $N\in \N$, the difference
\begin{equation}\label{FAsympExp}
 F(\lambda)- \sum_{j=0}^{N-1}\sum_{k=0}^{n_j} f_{jk}(\lambda) \log^k|\lambda|
\end{equation}
satisfies \eqref{FSymbolEstimate} with $\mu=-N+\eps$ for any $\eps>0$. 
Here $n_0=0$, and 
\[ f_{00}(\lambda)=t_0(\lambda)(A_\wedge-\lambda)\theta \] 
with $t_0(\lambda)$ as in \eqref{TAsympExp}. If $A$ has coefficients independent of $x$ near $\partial M$, then $n_j=0$ for all $j$.
\end{enumerate}
\end{proposition}
\begin{proof}
For $\lambda\in\Lambda_R$ let $\hat\lambda=\lambda/|\lambda|$. As a first step, we will show the existence of functions $C_{jk}\in C^\infty(S^1\cap\Lambda, \L(\Dom_{\max}/\Dom_{\min},\C^d))$ such that
\begin{equation*}
 F(\lambda)\tilde\kappa_{|\lambda|^{1/m}} \sim t_0(\hat\lambda)(A_\wedge-\hat\lambda)\theta + \sum_{j=1}^\infty |\lambda|^{-j/m}\sum_{k=0}^{n_j} C_{jk}(\hat\lambda)\log^k |\lambda|.
\end{equation*}
This asymptotic expansion will directly lead to the claimed properties of $F(\lambda)$ with coefficients $f_{jk}$ defined by
\begin{equation*}
 f_{jk}(\lambda)=|\lambda|^{-j/m} C_{jk}(\hat\lambda)\tilde\kappa_{|\lambda|^{1/m}}^{-1}
 \quad\text{for } \lambda\in\Lambda\minus\{0\}.
\end{equation*}
Observe that, since $F(\lambda)$ is tempered, it suffices to check \eqref{FSymbolEstimate} only for $\alpha=\beta=0$. Also, since $\Dom_{\max}/\Dom_{\min}= \bigoplus_{\sigma_0\in\Sigma} \tilde\Sing_{\sigma_0}$, it is sufficient to expand the restriction of $F(\lambda)\tilde \kappa_{|\lambda|^{1/m}}$ to $\tilde\Sing_{\sigma_0}$ for every $\sigma_0\in\Sigma$. 

Let $\varrho=|\lambda|^{1/m}$, $\ell\in\N$, $\sigma_0\in\Sigma$, and write
\begin{equation*}
 F(\lambda)\tilde \kappa_\varrho = T(\lambda)(A-\lambda)\Ka_\ell(\varrho): \tilde\Sing_{\sigma_0}\to \C^d.
\end{equation*}
We start by expanding the operator family
\begin{align*} 
\kappa_\varrho^{-1}(A-\lambda)\Ka_\ell(\varrho)= \kappa_\varrho^{-1}A \kappa_\varrho\omega_0 L^{(\ell)}_\varrho -\lambda\omega_0 L^{(\ell)}_\varrho. 
\end{align*}
In \cite[Lemma~6.18]{GKM2} we proved that for every $\vt\in J_{\sigma_0,0}$ and $\psi\in \Sing_{\wedge,\sigma_0}$ there is a polynomial $q_{\vt}(y,\log x,\log\varrho)$ in $(\log x, \log\varrho)$ with coefficients in $C^\infty(Y;E)$ such that
\[ \e_{\sigma_0,\vt}(\varrho)(\psi) = q_{\vt}(y,\log x,\log\varrho) x^{i(\sigma_0-i\vt)}. \]
In fact, there are operators $c_{\sigma_0,\vt,k}\in\L(\Sing_{\wedge,\sigma_0},C^\infty(\open Y^\wedge;E))$, and $\mu_\vt\in\N$, such that
\begin{equation}\label{esigmatheta}
 \e_{\sigma_0,\vt}(\varrho)=\sum_{k=0}^{\mu_\vt} c_{\sigma_0,\vt,k} \log^k\varrho
\end{equation}
for every $\vt\in J_{\sigma_0,\ell}$. Therefore, on $\tilde\Sing_{\sigma_0}$,
\begin{equation}\label{LrhoExpansion}
\lambda \omega_0 L^{(\ell)}_\varrho = \varrho^m\Big(\hat\lambda\omega_0\theta + \sum_{\vt\in J_{\sigma_0,\ell}} \varrho^{-\vt} \sum_{k=0}^{\mu_\vt} (\hat\lambda \omega_0 c_{\sigma_0,\vt,k}\,\theta) \log^k\varrho \Big).
\end{equation}
We split (near the boundary)
\begin{equation*}
 A=x^{-m}\sum_{\nu=0}^{\ell+m} P_\nu x^\nu + x^{\ell+1}\tilde P_\ell,
\end{equation*}
where $\tilde P_\ell\in \Diff_b^m(Y^\wedge;E)$, and where each $P_\nu\in \Diff_b^m(Y^\wedge;E)$ has coefficients independent of $x$. In particular, $\kappa_\varrho^{-1} P_\nu x^\nu \kappa_\varrho = \varrho^{-\nu} P_\nu x^\nu$, so 
\begin{equation*}
 \kappa_\varrho^{-1} A \kappa_\varrho=x^{-m}\sum_{\nu=0}^{\ell+m}\varrho^{m-\nu} P_\nu x^\nu + \varrho^{-\ell-1}x^{\ell+1} \kappa_\varrho^{-1}\tilde P_\ell\kappa_\varrho.
\end{equation*}
Now, as in the proof of \cite[Lemma~6.20]{GKM2}, one can show that the family of operators $\kappa_\varrho^{-1}\tilde P_\ell\kappa_\varrho\omega_0 L^{(\ell)}_{\varrho}$ is $O(1)$ in $\L(\tilde\Sing_{\sigma_0},x^{-m/2}L_b^2)$ as $\varrho\to\infty$. Therefore, modulo a remainder that is $o(\varrho^{-\ell})$ as $\varrho\to\infty$, we have
\begin{equation*}
\begin{split}
\big(\kappa_\varrho^{-1} A\kappa_{\varrho}\big) & \omega_0 L^{(\ell)}_{\varrho} \\
&\equiv x^{-m}\sum_{\nu=0}^{\ell+m}\varrho^{m-\nu} P_\nu x^\nu {\omega_0}L^{(\ell)}_{\varrho} \\
&= \varrho^m \Bigl(x^{-m}\sum_{\nu=0}^{\ell+m}\varrho^{-\nu} P_\nu x^\nu\Bigr) \omega_0 \Bigl(\theta + \sum_{\vt\in J_{\sigma_0,\ell}} \varrho^{-\vt} \e_{\sigma_0,\vt}(\varrho) \theta\Bigr) \\
&\equiv \varrho^m \Bigl(A_\wedge\omega_0 \theta + \sum_{j=1}^{\ell} \varrho^{-j} \Bigl(\sum_{\nu+\vt=j} x^{-m} P_\nu x^\nu \omega_0 \e_{\sigma_0,\vt}(\varrho)\theta \Bigr) \Bigr).
\end{split}
\end{equation*}
Because of \eqref{esigmatheta} and \eqref{LrhoExpansion}, there are $D_{jk}\in C^\infty(S^1\cap\Lambda, \L(\tilde\Sing_{\sigma_0},\K^{\infty,-m/2}))$ and $n_j\in\N_0$ such that
\begin{equation}\label{AKaExpansion}
 \kappa_\varrho^{-1}(A-\lambda)\Ka_\ell(\varrho) \equiv
 \varrho^m \Bigl((A_\wedge-\hat\lambda)\omega_0 \theta + \sum_{j=1}^{\ell} \varrho^{-j}\sum_{k=0}^{n_j} D_{jk}(\hat\lambda) \log^k|\lambda|\Bigr)
\end{equation}
modulo $o(\varrho^{-\ell})$ as $\varrho\to\infty$. The functions $D_{jk}$ are independent of the given $\ell$.

To complete the expansion of $F(\lambda)\tilde\kappa_\varrho$, let $\omega$ be an arbitrary  cut-off function and let $t(\lambda)=T(\lambda)\omega$. Since $T(\lambda)(1-\omega)$ is rapidly decreasing as $|\lambda|\to\infty$, for every $N\in\N$ we have
\begin{equation}\label{TAlambdaKExpansion}
\begin{split}
 T(\lambda)(A-\lambda)\Ka_\ell(\varrho) &= t(\lambda)(A-\lambda)\Ka_\ell(\varrho) + o(\varrho^{-N})\\
 &= \big(t(\lambda)\kappa_\varrho\big)\big(\kappa_\varrho^{-1}(A-\lambda)\Ka_\ell(\varrho)\big) + o(\varrho^{-N})
\end{split}
\end{equation}
as $\varrho\to\infty$. Now, by Proposition~\ref{Tlambda}, $t(\lambda)\kappa_\varrho$ admits an asymptotic expansion
\begin{equation*}
 t(\lambda)\kappa_\varrho \sim \sum_{j=0}^\infty t_j(\lambda)\kappa_\varrho = 
 \sum_{j=0}^\infty \varrho^{-m-j} t_j(\hat\lambda).
\end{equation*}
This expansion, combined with \eqref{AKaExpansion} and \eqref{TAlambdaKExpansion} for $\ell$ sufficiently large, gives the desired expansion of $F(\lambda)\tilde\kappa_{|\lambda|^{1/m}}$. Observe that, on $\Dom_{\max}/\Dom_{\min}$,
\begin{equation*}
t_0(\hat\lambda)(A_\wedge-\hat\lambda)\omega_0 \theta = t_0(\hat\lambda)(A_\wedge-\hat\lambda)\theta, 
\end{equation*}
and therefore $f_{00}(\lambda)= t_0(\hat\lambda)(A_\wedge-\hat\lambda)\theta \tilde\kappa_\varrho^{-1} = t_0(\lambda)(A_\wedge-\lambda)\theta$.

If $A$ has coefficients independent of $x$ near $\partial M$, then $\e_{\sigma_0,k}=0$ for every $\sigma_0$ and all $k\in\N$, $\theta_\ell^{-1}$ and $L_\varrho^{(\ell)}$ are the identity map, and $\Ka_\ell(\varrho)u=\kappa_\varrho\omega_0 u$. Thus
\begin{equation*}
 \kappa_\varrho^{-1}(A-\lambda)\Ka_\ell(\varrho) =
 \varrho^m (A_\wedge-\hat\lambda)\omega_0  
\end{equation*}
as opposed to \eqref{AKaExpansion}. Consequently, the expansion of $F(\lambda)$ given in \eqref{FAsympExp} has no $\log$ terms in this case. In other words, $n_j=0$ for all $j$.
\end{proof}

The following proposition gives the existence and structure of $F_\Dom(\lambda)^{-1}$ under the assumption that $\Dom$ is stationary. The nonstationary case will be treated in \cite{GKM5b}. 

\begin{proposition}\label{FInvStructure}
Under our general assumptions, there exists $R>0$ such that $F_\Dom(\lambda)=F(\lambda)|_{\Dom/\Dom_{\min}}$ is invertible for all $\lambda\in\Lambda_R$. The family $F_\Dom(\lambda)^{-1}$ has the following properties:
\begin{enumerate}[$(i)$]
\item $F_\Dom(\lambda)^{-1} \in C^\infty(\Lambda_R;\L(\C^d,\Dom_{\max}/\Dom_{\min}))$ and for every $\alpha,\beta\in\N_0$ we have
\begin{equation}\label{FInvSymbolEstimate}
\norm{\tilde\kappa_{|\lambda|^{1/m}}^{-1}\partial_\lambda^\alpha \partial_{\bar\lambda}^\beta\, F_\Dom(\lambda)^{-1}} = O(|\lambda|^{\frac{\mu}{m}-\alpha-\beta}) \;\text{ as }\; |\lambda|\to\infty,
\end{equation}
with $\mu=0$;
\item for all $j\in\N_0$ there exist $m_j\in\N_0$ and $\phi_{jk}\in C^\infty(\Lambda\minus\{0\},\L(\C^d,\Dom_{\max}/\Dom_{\min}))$, $k=0,\dots,m_j$, with
\begin{equation*}
 \phi_{jk}(\varrho^m\lambda)=\varrho^{-j}\tilde\kappa_\varrho
 \phi_{jk}(\lambda) \;\text{ for every } \varrho>0,
\end{equation*} 
such that for every $N\in \N$, the difference
\begin{equation}\label{FInvAsympExp}
 F_\Dom(\lambda)^{-1}- \sum_{j=0}^{N-1}\sum_{k=0}^{m_j} \phi_{jk}(\lambda) \log^k|\lambda|
\end{equation}
satisfies \eqref{FInvSymbolEstimate} with $\mu=-N+\eps$ for any $\eps>0$. Here $m_0=0$, and
\[ \phi_{00}(\lambda) = \big(f_{00}(\lambda)|_{\Dom/\Dom_{\min}}\big)^{-1} \]
with $f_{00}(\lambda)$ as in \eqref{FAsympExp}. If $A$ has coefficients independent of $x$ near $\partial M$, then $m_j=0$ for all $j$.
\end{enumerate}
\end{proposition}
\begin{proof}
As explained in Section~\ref{sec:ResolventStructure}, the invertibility of $A_{\wedge,\Dom_{\wedge}}-\lambda$ is equivalent to the invertibility of $F_\wedge(\lambda)$ on $\Dom_\wedge/\Dom_{\wedge,\min}$. Thus the restriction of $f_{00}(\lambda) = F_\wedge(\lambda)\theta$ to $\Dom/\Dom_{\min}$ is invertible. Let $\phi_{00} = \big(f_{00}(\lambda)|_{\Dom/\Dom_{\min}}\big)^{-1}$ and let
\begin{equation*}
 \Phi(\lambda)=F(\lambda)\phi_{00}(\lambda)=F_{\Dom}(\lambda)\phi_{00}(\lambda).
\end{equation*}
By Proposition~\ref{FStructure}, the family $F_{\Dom}(\lambda)$ admits the asymptotic expansion
\begin{equation*}
 F_{\Dom}(\lambda)\sim f_{00}(\lambda)|_{\Dom/\Dom_{\min}}+\sum_{j=1}^{\infty}\sum_{k=0}^{n_j} f_{jk}(\lambda)|_{\Dom/\Dom_{\min}} \log^k|\lambda|
\end{equation*}
in the sense that for every $N\in\N$, the difference \eqref{FAsympExp} satisfies the estimate  \eqref{FSymbolEstimate} with $\mu=-N+\eps$ for any $\eps>0$. 

Now, if $\varrho=|\lambda|^{1/m}$ and $\hat\lambda=\lambda/|\lambda|$, the $\tilde\kappa$-homogeneity of $f_{jk}$ implies
\begin{equation*}
 f_{jk}(\lambda)=f_{jk}(\varrho^m\hat\lambda)=\varrho^{-j}f_{jk}(\hat\lambda) \tilde\kappa_\varrho^{-1}. 
\end{equation*}
Moreover, since $\Dom$ is stationary and  $f_{00}$ is $\tilde\kappa$-homogeneous, we also have
\begin{equation*}
\phi_{00}(\lambda) = \tilde\kappa_\varrho \phi_{00}(\hat\lambda).
\end{equation*}
Thus $f_{jk}(\lambda)\phi_{00}(\lambda) = |\lambda|^{-j/m} f_{jk}(\hat\lambda)\phi_{00}(\hat\lambda)$, and so 
\begin{equation*}
 \Phi(\lambda)\sim 1+\sum_{j=1}^{\infty}|\lambda|^{-j/m} \sum_{k=0}^{n_j} f_{jk}(\hat\lambda)\phi_{00}(\hat\lambda) \log^k|\lambda|.
\end{equation*}
Hence there exist $m_j\in\N_0$, $m_j\le n_j$, and $D_{jk}:S^1\cap\Lambda\to \L(\C^d,\Dom_{\max}/\Dom_{\min})$, $k=0,\dots,m_j$, such that
\begin{equation*}
 \Phi(\lambda)^{-1}\sim 1+\sum_{j=1}^{\infty}|\lambda|^{-j/m} \sum_{k=0}^{m_j} D_{jk}(\hat\lambda)\log^k|\lambda|.
\end{equation*}
Since $F_{\Dom}(\lambda)^{-1}=\phi_{00}(\lambda)\Phi(\lambda)^{-1}$, we define 
\begin{equation*}
 \phi_{jk}(\lambda)= |\lambda|^{-j/m} \phi_{00}(\lambda) D_{jk}(\hat\lambda) \;
 \text{ for $j\in\N$ and $k\in\{0,\dots,m_j\}$}.
\end{equation*}
These functions are $\tilde\kappa$-homogeneous of order $-j$ and
\begin{equation*}
 F_{\Dom}(\lambda)^{-1}\sim \phi_{00}(\lambda)+\sum_{j=1}^{\infty} \sum_{k=0}^{m_j} \phi_{jk}(\lambda)\log^k|\lambda|
\end{equation*}
in the sense that, for every $N\in\N$, the difference \eqref{FInvAsympExp} satisfies \eqref{FInvSymbolEstimate} for $\alpha=\beta=0$ with $\mu=-N+\eps$ for any $\eps>0$. The corresponding estimates for $\alpha, \beta\in\N$ follow immediately. Observe that if $n_j=0$ for all $j$, then $m_j=0$ for all $j$.
\end{proof}

\begin{proposition}\label{PStructure}
Let $\vp\in C^\infty(M;\textup{End}(E))$ and let $\omega\in C_c^\infty([0,1))$ be an arbitrary cut-off function.  The family $P(\lambda)=\vp[1-B(\lambda)(A-\lambda)]$, interpreted as in \eqref{Plambda}, has the following properties:
\begin{enumerate}[$(i)$]
\item For every $s\in\R$ we have $(1-\omega)P(\lambda)\in \S(\Lambda_R,\L(\Dom_{\max}/\Dom_{\min},x^{-m/2}H_b^s))$ and $\omega P(\lambda) \in C^\infty(\Lambda_R,\L(\Dom_{\max}/\Dom_{\min},\K^{s,-m/2}))$;
\end{enumerate}
in addition, with $p(\lambda)=\omega P(\lambda)$, 
\begin{enumerate}[$(i)$] \stepcounter{enumi}
\item for every $\alpha,\beta\in\N_0$ we have
\begin{equation}\label{pSymbolEstimate}
\norm{\kappa_{|\lambda|^{1/m}}^{-1}\big(\partial_\lambda^\alpha \partial_{\bar\lambda}^\beta\, p(\lambda)\big)\tilde\kappa_{|\lambda|^{1/m}}} = O(|\lambda|^{\frac{\mu}{m}-\alpha-\beta}) \;\text{ as }\; |\lambda|\to\infty,
\end{equation}
with $\mu=0$;
\item for every $j\in\N_0$ there are $m_j\in\N_0$ and operator-valued functions $p_{jk}\in C^\infty(\Lambda\minus\{0\},\L(\Dom_{\max}/\Dom_{\min},\K^{s,-m/2}))$, $k=0,\dots,m_j$, with
\begin{equation*}
 p_{jk}(\varrho^m\lambda)=\varrho^{-j}\kappa_\varrho
 p_{jk}(\lambda)\tilde\kappa_{\varrho}^{-1} \;\text{ for every } \varrho>0,
\end{equation*} 
such that for every $N\in \N$, the difference
\begin{equation}\label{pAsympExp}
 p(\lambda)- \sum_{j=0}^{N-1}\sum_{k=0}^{m_j} p_{jk}(\lambda) \log^k|\lambda|
\end{equation}
satisfies \eqref{pSymbolEstimate} with $\mu=-N+\eps$ for any $\eps>0$. Here $m_0=0$ and 
\[ p_{00}(\lambda)= \vp_0[1-B_\wedge(\lambda)(A_\wedge-\lambda)]\theta. \]
If $A$ has coefficients independent of $x$ near $\partial M$, then $m_j=0$ for all $j$.
\end{enumerate}
\end{proposition}
\begin{proof}
Let $\omega_0$ and $\omega_1$ be cut-off functions such that $\omega_0\prec \omega_1\prec \omega$. Since the operator $(1-B(\lambda)(A-\lambda))$ vanishes on $\Dom_{\min}$ for $\lambda\in\Lambda_R$, and since $A$ is local, we have 
\begin{align*}
 (1-\omega)\vp(1-B(\lambda)(A-\lambda))&= (1-\omega)\vp(1-B(\lambda)(A-\lambda))\omega_0 \\
 &= -\vp\big[(1-\omega) B(\lambda)\omega_1\big] (A-\lambda)\omega_0.
\end{align*}
The fact that $(1-\omega) B(\lambda)\omega_1$ belongs to $\S(\Lambda_R,\L(x^{-m/2}L_b^2,x^{-m/2}H_b^s))$ for every $s\in\R$ implies the claimed decay of $(1-\omega)P(\lambda)$ as $|\lambda|\to\infty$.

Let $\varrho=|\lambda|^{1/m}$, $\ell\in\N$, and assume $\varrho\ge 1$. As in the proof of Proposition~\ref{FStructure} we can use the lift $\Ka_{\ell}(\varrho)=\kappa_\varrho\omega_0 L^{(\ell)}_\varrho$, see \eqref{Lrho&Ktilde}, and write
\begin{equation*}
 p(\lambda) \tilde\kappa_{\varrho}= \omega\vp\big(1-B(\lambda)(A-\lambda)\big)\Ka_{\ell}(\varrho): 
\Dom_{\max}/\Dom_{\min}\to \K^{s,-m/2}(Y^\wedge;E).
\end{equation*}
Then, since $(1-\omega_1)\omega_0=0$, we have
\begin{equation*}
\begin{split}
 \kappa_{\varrho}^{-1} p(\lambda) \tilde\kappa_{\varrho}
 &= \kappa_{\varrho}^{-1}\omega\vp \Ka_{\ell}(\varrho) - \kappa_{\varrho}^{-1}\omega\vp B(\lambda)(A-\lambda) \Ka_{\ell}(\varrho) \\
 &= \kappa_{\varrho}^{-1}\big[\omega\vp\omega_0(x\varrho)\big] \kappa_\varrho L^{(\ell)}_\varrho - \kappa_{\varrho}^{-1}\big[\omega\vp B(\lambda)\omega_1(x\varrho)\big]\kappa_\varrho \kappa_{\varrho}^{-1}(A-\lambda) \Ka_{\ell}(\varrho). 
\end{split}
\end{equation*}
Setting $q_0(\lambda)=\omega\vp\,\omega_0(x|\lambda|^{1/m})$ and $q_1(\lambda)= \omega\vp B(\lambda)\omega_1(x|\lambda|^{1/m})$, we get
\begin{equation*}
\kappa_{\varrho}^{-1} p(\lambda) \tilde\kappa_{\varrho}= 
\big(\kappa_{\varrho}^{-1}q_0(\lambda) \kappa_\varrho\big) L^{(\ell)}_\varrho - \big(\kappa_{\varrho}^{-1}q_1(\lambda)\kappa_\varrho\big) \kappa_{\varrho}^{-1}(A-\lambda) \Ka_{\ell}(\varrho). 
\end{equation*}

We start by expanding $q_0(\lambda)$. Since $(1-\omega)\omega_0(x\varrho)=0$ for $\varrho\ge 1$, we have
\begin{equation*}
 q_0(\lambda)=\omega\vp\,\omega_0(x|\lambda|^{1/m})=\vp\,\omega_0(x|\lambda|^{1/m}).
\end{equation*}
Without loss of generality we can assume $\vp$ to be supported in a collar neighborhood of $\partial M=Y$ and consider it a function in $C^\infty([0,1), C^\infty(Y;\textup{End}(E|_Y)))$. Thus it has a Taylor expansion
\begin{equation*}
 \vp \sim \sum_{\nu=0}^\infty \vp_\nu(y) x^\nu
\end{equation*}
with $\vp_\nu\in C^\infty(Y;\textup{End}(E|_Y))$.  Since $\vp_\nu(y) x^\nu= \varrho^{-\nu}\kappa_\varrho \vp_\nu(y) x^\nu \kappa_\varrho^{-1}$, we get 
\begin{equation}\label{q0Expansion}
 \kappa_{\varrho}^{-1}q_0(\lambda) \kappa_\varrho \sim
 \sum_{\nu=0}^\infty \varrho^{-\nu} \vp_\nu x^\nu \omega_0.
\end{equation}

On the other hand, 
\begin{equation*}
 q_1(\lambda) \sim \sum_{j=0}^{\infty} q_{1,j}(\lambda)
\end{equation*}
in the sense of Lemma~\ref{BSymbolExpansion}. In particular, $q_{1,j}\in C^\infty(\Lambda\minus\{0\},\ell^1(\K^{s,-m/2}))$ for $j\in\N_0$, and 
$q_{1,j}(\varrho^m\lambda)=\varrho^{-m-j} \kappa_\varrho\, q_{1,j}(\lambda)\kappa_\varrho^{-1}$ for every $\varrho>0$. Therefore,
\begin{equation*}
 \kappa_{\varrho}^{-1}q_1(\lambda) \kappa_\varrho \sim 
 \sum_{j=0}^{\infty} \varrho^{-m-j} q_{1,j}(\hat\lambda)
\end{equation*}
with $q_{1,0}(\hat\lambda)=\vp_0 B_\wedge(\hat\lambda) \omega_1(x)$.
This, together with \eqref{AKaExpansion}, \eqref{q0Expansion}, and \eqref{LrhoExpansion} (without $\lambda\omega_0$), leads to the expansion \eqref{pAsympExp}. Observe that, on $\Dom_{\max}/\Dom_{\min}$,
\begin{equation*}
 \vp_0(1-B_\wedge(\hat\lambda)(A_\wedge-\hat\lambda))\omega_0\theta=
 \vp_0(1-B_\wedge(\hat\lambda)(A_\wedge-\hat\lambda))\theta,
\end{equation*}
and therefore, the principal component $p_{00}(\lambda)$ is as claimed.

For the same reasons as in the proof of Proposition~\ref{FStructure}, if $A$ has coefficients independent of $x$ near $\partial M$, there are no $\log$ terms in the expansion of $p(\lambda)$.
\end{proof}

{\bf Acknowledgments.} 
We are very grateful to the Mathematisches Forschungsinstitut Oberwolfach for their excellent support with the ``Research in Pairs'' program. The research presented in this paper was initiated during our stay at the institute.


\end{document}